\documentclass[11pt,a4paper]{amsart}
\usepackage[]{fontenc}
\usepackage[latin1]{inputenc}
\usepackage{graphicx}
\usepackage{epstopdf}

\makeatletter
 \theoremstyle{plain}
\newtheorem{thm}{Theorem}[section]
  \theoremstyle{remark}
  \newtheorem{rem}[thm]{Remark}
  \theoremstyle{plain}
  \newtheorem{cor}[thm]{Corollary}
  \theoremstyle{plain}
  \newtheorem{prop}[thm]{Proposition}
  \theoremstyle{plain}
  \newtheorem{lem}[thm]{Lemma}
  \theoremstyle{plain}
  \newtheorem{dfn}[thm]{Definition}
  \theoremstyle{remark}
  \newtheorem{ex}[thm]{Example}

\usepackage[all]{xy}
\usepackage{bbm}

\def\al{\alpha}

\def\ga{\gamma}
\def\de{\delta}

\def\ep{\varepsilon}

\def\la{\lambda}

\def\si{\sigma}
\def\th{\vartheta}
\def\vp{\varphi}

\def\om{\omega}

\def\dd{{\mathrm d}}
\def\e{{\mathrm e}}
\def\ii{{\mathrm i}}
\def\t{{}^{t}}
\def\A{{\mathcal A}}
\def\Atrop{{\mathcal A}_{\rm trop}}
\def\Alim{{\mathcal A}_{\rm lim}}
\def\C{{\mathbb C}}
\def\CC{{\mathcal C}}

\def\L{{\mathcal L}}
\def\P{{\mathcal P}}

\def\iP{\breve{P}}
\def\Q{{\mathcal Q}}
\def\R{{\mathbb R}}
\def\S{{\mathbb S}}
\def\t{{}^t}
\def\T{{\mathbb T}}
\def\Z{{\mathbb Z}}
\def\pii{\pi\mathrm{i}}

\def\u1{{\mathbbm 1}^{U(1)}}
\def\bu1{\breve{{\mathbbm 1}}^{U(1)}}
\def\baru1{\ol{{\mathbbm 1}}^{U(1)}}

\newenvironment{proofb}[1]{\vspace{2ex} \noindent \textbf{\textit{Proof of Theorem} #1.}}%
{\qed}

\def\tns{\otimes}
\def\ol{\overline}
\newcommand{\pd}[2]{\frac{\partial #1}{\partial #2}}

\def\spanC{\mathrm{span}_{\mathbb{C}}\,}
\def\nn{\nonumber}
\newcommand{\cdiag}[1]{\begin{array}{c} \xymatrix{ #1 } \end{array}}

\newcommand{\be}{\begin{equation}}
\newcommand{\ee}{\end{equation}}
\newcommand{\ba}{\begin{eqnarray}}
\newcommand{\ea}{\end{eqnarray}}
\def\dist{\mathrm{dist}}
\newcommand{\vect}[1]{\left[\begin{array}{@{}c@{}} #1 \end{array}\right]}
\def\idmax{\left[\begin{array}{@{}cc@{}} 1 & 0 \\ 0 & 1 \end{array}\right]}
\def\gmax{\left[\begin{array}{@{}cc@{}} 6 & 3 \\ 3 & 6 \end{array}\right]}

\def\Hess{\mathrm{Hess}\,}
\def\sgn{\mathrm{sign}}
\def\Vol{{\rm Vol}}

\usepackage[english]{babel}

\begin{document}

\title[Large toric K\"ahler metrics, quantization and amoebas]{Toric K\"ahler metrics seen from infinity, quantization and compact tropical amoebas}

\author[T.~Baier]{Thomas Baier}
\address{Thomas Baier\\Centro de Matem\'{a}tica da Universidade do Porto, R. do Campo Alegre 687, 4169-007 Porto, Portugal}
\email{tbaier@fc.up.pt}
\author[C.~Florentino]{Carlos Florentino}
\author[J.M.~Mour\~{a}o]{Jos\'{e} M. Mour\~{a}o}
\author[J.P.~Nunes]{Jo\~{a}o P. Nunes}
\address{Carlos Florentino, Jos\'{e} M. Mour\~{a}o, Jo\~{a}o P. Nunes\\Department of Mathematics, Instituto Superior T\'{e}cnico, Av. Rovisco Pais, 1049-001 Lisboa, Portugal}
\email{cfloren, jmourao, jpnunes@math.ist.utl.pt}

\begin{abstract}
We consider the metric space of all toric K\"ahler metrics on a compact toric manifold; when ``looking at it from infinity'' (following Gromov), we obtain the tangent cone at infinity, which is parametrized by equivalence classes of complete geodesics. In the present paper, we study the associated limit for the family of metrics on the toric variety, its quantization, and degeneration of generic divisors.

The limits of the corresponding K\"ahler polarizations become degenerate along the Lagrangian fibration defined by the moment map. This
allows us to interpolate continuously between geometric quantizations in the
holomorphic and real polarizations and show that the monomial
holomorphic sections of the prequantum bundle converge to Dirac
delta distributions supported on Bohr-Sommerfeld fibers. 

In the second part, we use these families of toric metric
degenerations to study the limit of compact hypersurface amoebas and
show that in Legendre transformed variables they are described by
tropical amoebas.
We believe that
our approach gives a different, complementary, perspective on the
relation between complex algebraic geometry and tropical geometry.
\end{abstract}

\maketitle

\tableofcontents{}


\section{Introduction and main results}

Studying families of
toric K\"ahler metrics on a smooth toric variety, we investigate
limits corresponding to holomorphic
Lagrangian distributions degenerating to the real Lagrangian torus
fibration defined by the moment map.
We use methods of K\"ahler geometry and geometric quantization, which permits us to consider degenerations even though the algebraic-geometric moduli space of complex structures associated to toric varieties consists of a point only.
More precisely (see Section \ref{sect_tangentcone}), we consider the space of all toric K\"ahler metrics on a fixed very ample toric line bundle, and the limits we take along complete geodesics are parametrized in a natural way by the tangent cone at infinity of this space. Below, we study the associated limit for the corresponding family of metrics on the toric variety, its quantization, and degeneration of generic divisors. While the metric limits are distinct across the tangent cone at infinity, the limit lagrangian foliation is the same for all points.

This approach permits us to obtain the following main results. Let $P$ be a Delzant polytope and let $X_{P}$ be the associated compact toric variety \cite{De}. Let $\psi\in C^{\infty}(P)$ be a smooth
function with positive definite Hessian on $P$. Such a $\psi$ defines
a complete geodesic in the space of toric K\"ahler metrics (see Section \ref{sect_tangentcone}). Then,
\renewcommand{\labelenumi}{\bf\arabic{enumi}.}

\begin{enumerate}
\item In Theorem \ref{th1b}, we determine the weakly covariantly constant
sections of the natural line bundle on $X_P$ with respect to the (singular)
real polarization defined by the Lagrangian fibration given by the moment map
$\mu_P$. In particular, we see that they are naturally indexed by the
integer points in the polytope $P$.
\item We show in Theorem \ref{th1} that the family of K\"ahler polarizations
corresponding to the mentioned geodesic converges to the real polarization,
independently of the direction $\psi$ of deformation.
\item Theorem \ref{th2} states that the holomorphic monomial sections of the
natural line bundle converge to the Dirac delta distributions
supported on the corresponding Bohr-Sommerfeld orbit of Theorem \ref{th1b}.

{}For the class of symplectic toric manifolds,
this solves the important question, in the context of geometric
quantization, on the explicit link between K\"ahler polarized Hilbert
spaces and real polarized ones, in particular in a situation where
the real polarization is singular. For a different, but related, recent
result in this direction see \cite{BGU}.
\item We show that the compact amoebas \cite{GKZ,FPT,Mi} of complex hypersurface
varieties in $X_{P}$ converge in the Hausdorff metric to tropical
amoebas in the ($\psi$-dependent) variables defined from the symplectic
ones via the Legendre transform
\begin{eqnarray}
 {\L}_{\psi}\ :\ P & \rightarrow & {\L}_{\psi}P\subset\R^{n}\nonumber \\
u & = & {\L}_{\psi}(x)=\frac{\partial\psi}{\partial x}(x).\label{ltr}
\end{eqnarray}
This framework
gives a new way of obtaining tropical geometry from complex algebraic
geometry by degenerating the ambient toric metric rather
than taking a limit of deformations of the complex field \cite{Mi,EKL}.

Another significative difference is that the limit amoebas
described above live inside the compact image $\L_{\psi} P$ and
are tropical in the interior of $\L_\psi P$.
\end{enumerate}
Let us describe these results in more detail.

\subsection{Geometric quantization of toric varieties}
Let $P$ be a Delzant
polytope with vertices in $\Z^{n}$ defining, via the Delzant construction
\cite{De}, a compact symplectic toric manifold $(X_{P},\omega,\T^{n},\mu_{P})$,
with moment map $\mu_{P}$. Let $\P_{\R}\subset(TX_{P})_{\C}$
be the (singular) real polarization, in the sense of geometric quantization
\cite{Wo}, corresponding to the orbits of the Hamiltonian $\T^{n}$
action. The Delzant construction also defines a complex structure
$J_{P}$ on $X_{P}$ such that the pair $(\omega,J_{P})$, is
K\"ahler, with K\"ahler polarization $\P_{\C}$. In addition, the polytope
$P$ defines, canonically, an equivariant $J_{P}$-holomorphic
line bundle, $L \rightarrow X_{P}$ with curvature $-i\om$
\cite{Od}.

A result, usually attributed to Danilov and Atiyah \cite{Da,GGK},
states that the number of integer points in $P$, which are equal
to the images under $\mu_{P}$ of the Bohr-Sommerfeld (BS) fibers
of the real polarization $\P_{\R}$, is equal to the number of holomorphic
sections of $L$, i.e. to the dimensionality of $H^{0}(X_{P},L)$.

An important general problem in geometric quantization is understanding
the relation between quantizations associated to different polarizations
and, in particular, between real and holomorphic quantizations. Hitchin
\cite{Hi} has shown that, in some general situations, the bundle
of quantum Hilbert spaces over the space of deformations of the complex
structure, is equipped with a (projectively) flat connection that
provides the identification between holomorphic quantizations corresponding to
different complex structures. These results do not, however, directly apply in the present situation as the complex structure on a toric variety is rigid. Concerning real polarizations, \'{S}niatycki \cite{Sn} has shown that for
non-singular real polarizations of arbitrary (quantizable) symplectic
manifolds, the set of BS fibers is in bijective correspondence with
a generating set for the space of cohomological wave functions which
define the quantum Hilbert space in the real polarization. Explicit
geometro-analytic relations between real polarization wave
functions and holomorphic ones via degenerating families of complex structures have been found for theta functions on abelian varieties (see \cite{FMN,BMN} and references therein).
Similar studies have been performed for cotangent bundles of Lie groups \cite{Hal,FMMN}.
Some of the results in this paper, in fact, are related to these results for
the case $(T^{*}\S^{1})^{n}=(\C^{*})^{n}$, where in the present setting 
$(\C^{*})^{n}$
becomes the open dense orbit in the toric variety $X_{P}$.

As opposed to all these cases, however, the real polarization of a
compact toric variety always contains singular fibers. As was shown
by Hamilton \cite{Ham}, the sheaf cohomology used by \'{S}niatycki
only detects the non-singular BS leaves. Also, a possible model for
the real quantization that includes the singular fibers has recently
been described in \cite{BGU}.

If, on one hand, it is natural to expect that by finding a family of
(K\"ahler) complex structures degenerating to the real
polarization, the holomorphic sections will converge to delta distributions
supported at the BS fibers, on the other hand, it was unclear how
to achieve such behavior from the simple monomial sections characteristic
of holomorphic line bundles on toric manifolds (where the series
characteristic of theta functions on Abelian varieties are absent).

The detailed study of the degenerating K\"ahler structures and their quantization is made possible by Abreu's description
of toric complex structures \cite{Ab1,Ab2}, following Guillemin's
characterization of a canonical toric K\"ahler metric on $(X_P,\omega)$ determined by a symplectic potential $g_P: P\to\mathbb{R}$ \cite{Gui}. In particular, for any pair of smooth functions $\varphi,\psi$ satisfying certain convexity conditions (see Section \ref{symppots}), the functions $g_s$
\begin{equation}
 s \mapsto g_{s}=g_{P}+\varphi+s\psi,\label{gg}
\end{equation}
are admissible as symplectic potentials, i.e. define toric K\"ahler metrics for all positive $s$.

The quantization of a compact symplectic manifold in the real polarization is given 
by distributional sections.
In this case, conditions of covariant constancy of the wave functions have
to be understood as  local rather than pointwise (see, for instance, \cite{Ki})  
and the relevant piece of data is
the sheaf of smooth local sections of a polarization $\P$, $C^{\infty}(\P)$. 

Applying this approach just outlined, in Section \ref{realpol} we obtain a
description of the quantum space of the real polarization $\Q_\R$.
The main technical difference
as compared with the techniques of \cite{Sn} (applied to the present
situation in \cite{Ham}) consists in the fact that we do not only use
the sheaf of sections in the kernel of covariant differentiation, but
also the cokernel.
It is therefore not surprising that our result
differs from that of \cite{Ham}, where the dimension is given by the
number of non-degenerate Bohr-Sommerfeld fibers only. In contrast to
this, we find

\begin{thm} \label{th1b} For the singular real polarization $\P_\R$ defined by the
moment map, the space of covariantly constant distributional sections of the
prequantum line bundle $L_\om$ is spanned by one section $\de^m$ per Bohr-Sommerfeld
fiber $\mu_P^{-1}(m), m\in P\cap\Z^n$, with
\[
 \mathrm{supp}\,\de^m = \mu_P^{-1}(m) .
\]
\end{thm}

But not only does the result of the quantization in the real polarization
change; actually, the weak equations of covariant constancy allow for a
continuous passage from quantization in complex to real polarizations. The
first step in this direction is to verify that the conditions
imposed on distributional sections by the set of equations of covariant
constancy converge in a suitable sense: if we denote by $\P_{\C}^{s}$ the
holomorphic polarization corresponding to the complex structure defined by
(\ref{gg}), our second main finding is

\begin{thm}
\label{th1} For any $\psi\in C^\infty_{\textrm{Hess} > 0}(P)$, we have
\[
C^{\infty}(\lim_{s\rightarrow\infty}\P_{\C}^{s})=C^{\infty}(\P_{\R}),\]
where the limit is taken in the positive Lagrangian Grassmannian
of the complexified tangent space at each point in $X_{P}$. 
\end{thm}

Identifying holomorphic sections with distributional sections in the usual 
way (as in \cite{Gun}, but making use of the Liouville measure on the base) we
may actually keep track of the monomial basis of holomorphic sections as $s$
changes and show that they converge in the space of distributional sections.

Consider the prequantum bundle $L_\om$ 
equipped with the holomorphic structure defined by the prequantum connection
$\nabla$, defined in (\ref{conn2}), and by the complex structure $J_{s}$ corresponding
to $g_s$ in (\ref{gg}). Let  $\iota:C^\infty(L_\om)\to (C^\infty_c(L_\om^{-1}))'$ be the
natural injection of the space of smooth into distributional sections defined in
(\ref{iota}). For any lattice point  $m\in P\cap \Z^n$, let
$\sigma^m_s\in C^\infty(L_\om)$ be the associated $J_s-$holomorphic section 
of $L_\om$ and $\delta^m$ the delta distribution from the previous Theorem.
Our third main result is the following: 

\begin{thm}
\label{th2} 
For any $\psi$ strictly convex in a neighborhood of $P$
and $m\in P\cap\Z^{n}$, consider the family of $L^1-$normalized $J_s$-holomorphic sections
\[
 \cdiag{
 \R^{+}\ni s\mapsto\xi_{s}^{m}:=\frac{\si_{s}^{m}}{\|\si_{s}^{m}\|_{1}}
\in C^{\infty}(L_\om) \!\!\!\! \ar@{}[r]|-{\subset}^-{\iota} & \!\!\!\! (C^\infty_c(L_\om^{-1}))' }.
\]
Then, as $s\to\infty$, $\iota(\xi_s^m)$ converges to $\de^{m}$ in $(C^\infty_c(L_\om^{-1}))'$.
\end{thm}

\begin{rem}
Note that the sections $\sigma_{s}^{m},\sigma_{s}^{m'}$ are $L^{2}-$orthogonal
for $m\neq m'$. 
\end{rem}

\begin{rem}
We note that
the set up above can be easily generalized to a larger family of deformations given 
by symplectic potentials  of the form $g_s=g_P+\varphi +\psi_s$, where $\psi_s$ is a family of smooth strictly 
convex functions on $P$, such that $\frac{1}{s}\psi_s$ has a strictly convex limit in the 
$C^2$-norm in $C^\infty (P)$. 
\end{rem}

\begin{rem}
For non-compact symplectic toric manifolds $X_P$, the symplectic potentials in (\ref{gg})
still define compatible complex structures on $X_P$; however, Abreu's theorem no longer
holds. Theorems \ref{th1} and \ref{th2} remain valid in the non-compact case,
if one assumes uniform strict convexity of $\psi$ for the latter.
\end{rem}

As mentioned above, these results provide a setup for relating
quantizations in different polarizations. In particular, Theorem \ref{th2} 
gives an
explicit analytic relation between holomorphic and real wave functions
by considering families of complex structures converging to a degenerate
point.

\subsection{Compact tropical amoebas}

Let now $Y_{s}$ denote the one-parameter family of hypersurfaces
in $(X_{P},J_{s})$ given by
\begin{equation}
 \label{Ys} Y_{s}=\left\{ p\in X_{P}\ :\ \sum_{m\in P\cap\Z^{n}}a_{m}\e^{-sv(m)}\si_{s}^{m}(p)=0\right\} ,
\end{equation}
where 
$a_{m}\in\C^{*},v(m)\in\R,\forall m\in P\cap\Z^{n}$.
The image of $Y_{s}$ in $P$ under the moment map $\mu_{P}$ is naturally
called the compact amoeba of $Y_{s}$. Note that $Y_s$ is a complex submanifold
of $X_P$ equipped with the K\"ahler structure $(J_s,\gamma_s=\omega(.,J_s .))$.
This is in contrast with the compact amoeba of \cite{GKZ,FPT,Mi} where the
K\"ahler structure is held fixed.

Using the family of Legendre transforms  in (\ref{ltr}) associated to the potentials $g_s$ on
the open orbit, we relate the intersection $\mu_P(Y_s)\cap\iP$
with the $Log_t$-amoeba of \cite{FPT,Mi} for finite $t=\e^s$. For $s\to\infty$,
the Hausdorff limit of the compact amoeba is then characterized by
the tropical amoeba $\Atrop$ defined as the support of non-differentiability, or
corner locus, of the piecewise smooth continuous function $\R^n\to \R$,
\begin{eqnarray}
\nn u & \mapsto & \max_{m\in P\cap\Z^{n}}\left\{ \t m\cdot u-v(m)\right\} .
\end{eqnarray}

In Section \ref{cpamoeba}, we show that, as $s\to\infty$, the amoebas $\mu_P(Y_s)$
converge in the Hausdorff metric to a limit amoeba $\Alim$. The relation
between $\Atrop$ and $\Alim$ is given by a projection $\pi:\R^n\to\L_\psi P$
(defined in Lemma \ref{proj}, see Figure \ref{fig1}) determined by $\psi$ and the combinatorics of
the fan of $P$. The fourth main result is, then:

\begin{thm} \label{amoeba_thm} The limit amoeba is given by
\[
 \L_\psi \Alim = \pi \Atrop .
\]
\end{thm}

\begin{rem}
\label{55} There is a set with non-empty interior of valuations $v(m)$ in (\ref{Ys}), the convex
projection $\pi \Atrop$ coincides with the intersection
of $\Atrop$ with the image of the moment polytope $\L_\psi P$. In 
particular, if $\psi(x)=\frac{x^{2}}{2}$,
then $\L_{\psi}={\rm Id}_{P}$ and $\Alim$ is a (compact part) of a tropical
amoeba, see Figure \ref{fig2}.
\end{rem}

\begin{rem} Note that for quadratic $\psi(x)=\frac{\t xGx}{2}+\t bx$, where
 $\t G = G > 0$, the limit amoeba $\Alim \subset P$ itself is piecewise linear.
\end{rem}

Under certain conditions concerning $\psi$, this construction produces naturally a singular affine manifold $\L_\psi\Alim$,
with a metric structure (induced from the inverse of the Hessian of $\psi$).
In the last Section we comment on the possible relation of this result
to the study of  mirror symmetry from the SYZ viewpoint.

\section{Preliminaries and notation}
\label{prelim}

Let us briefly review a few facts concerning compatible complex structures
on toric symplectic manifolds and also fix some notation. For reviews on toric
varieties, see \cite{Co,Da,dS}. Consider a Delzant lattice
polytope $P\subset\R^{n}$ given by \begin{equation}
P=\left\{ x\in\R^{n}\ :\ \ell_{r}(x)\geq0,r=1,\dots,d\right\} ,\label{dp}\end{equation}
 where \begin{eqnarray}
\ell_{r}\ :\ \R^{n}\  & \rightarrow & \ \R\nonumber \\
\ell_{r}(x) & = & \t\nu_{r} x-\lambda_{r},\end{eqnarray}
$\lambda_{r}\in\Z$ and $\nu_{r}$ are primitive vectors of the lattice
$\Z^{n}\subset\R^{n}$, inward-pointing and normal to the $r$-th facet, i.e.
codimension-$1$ face of $P$. We denote the interior of $P$ by $\iP$, and the convex hull of $k$ points $v_1,\dots,v_k$ by $\langle v_1,\dots,v_k \rangle$.

Let $X_{P}$ be the associated smooth toric variety, with moment map
$\mu_P:X_P\to P$. On the open dense orbit $\breve X_P=\mu_P^{-1}(\iP)
\cong\iP\times\T^{n}$, one considers symplectic, or action-angle, coordinates
$(x,\theta)\in\iP\times\T^{n}$ for which the symplectic form is the
standard one, $\omega =\sum_{i=1}^{n}dx_{i}\wedge d\theta_{i}$ and $\mu_P(x,\theta)=x$.

\subsection{Symplectic potentials for toric K\"ahler structures}\label{symppots} Recall (\cite{Ab1,Ab2}) that any compatible complex structure on $X_P$ can
be written via a symplectic potential $g=g_{P}+\varphi$, where $g_{P}\in
C^{\infty}(\iP)$ is given by \cite{Gui}
\begin{equation}
g_{P}(x)=\frac{1}{2}\sum_{r=1}^{d}\ \ell_{r}(x)\log\ell_{r}(x),\label{gpp}
\end{equation}
and $\vp$ belongs to the convex set,  
$C_{g_{P}}^{\infty}(P)\subset C^{\infty}(P)$,
of functions $\vp$ such that ${\rm Hess}_{x}(g_{P}+\vp)$
is positive definite on $\breve P$ and satisfies the regularity
conditions
\begin{equation} \label{detHess}
{\rm det}({\rm Hess}_{x}(g_{P}+\vp))=\left[\alpha(x)\Pi_{r=1}^{d}\ell_{r}(x)\right]^{-1},
\end{equation}
for $\alpha$ smooth and strictly positive on $P$, as described
in \cite{Ab1,Ab2}.

The complex structure $J$ associated to a potential $g=g_P+\vp$ and the K\"ahler metric 
$\gamma=\omega(\cdot, J\cdot)$ are given by 
\begin{equation}
J=\left(\begin{array}{rl}
0 & -G^{-1}\\
G & 0\end{array}\right)\ ;\ \gamma=\left(\begin{array}{rl}
G & 0\\
0 & G^{-1}\end{array}\right),\label{cs}
\end{equation}
where $G=\Hess_x g > 0$ is the Hessian of $g$.
For recent applications of this result see e.g. \cite{Do,MSY,SeD}.

The complex coordinates are related with the symplectic ones by a
bijective Legendre transform \[
\iP\ni x\mapsto y=\frac{\partial g}{\partial x}\in\R^{n},\]
that is,  $g$ fixes an equivariant biholomorphism $\iP\times\T^{n}\cong(\C^{*})^{n}$, \[
 \iP\times\T^n \ni (x,\theta) \mapsto w=\e^{y+\ii \theta} \in (\C^*)^n.
\] 
The inverse transformation is given by $x=\frac{\partial h}{\partial y}$, where
\begin{equation}
\label{herm}
h(y)=\t x(y)y-g(x(y)).
\end{equation}

Let us describe coordinate charts covering the rest of $X_P$, that is, the loci of
compactification. 
Using the Legendre transform associated to 
the symplectic potential $g$, one can  
describe, in particular, holomorphic charts around the fixed points of
the torus action. Consider, for any vertex $v$ of $P$,
any ordering of the $n$ facets that contain $v$; upon reordering the
indices, one may suppose that
$$
 \ell_1(v) = \dots = \ell_n(v) = 0.
$$
Consider the affine change of variables on $P$
$$
 l_i = \ell_i(x) = \t\nu_i x-\lambda_i, \quad \forall i=1,\dots,n, \textrm{ i.e. }
 \quad l = Ax-\lambda,
$$
where $A=(A_{ij}=(\nu_i)_j)\in Gl(n,\Z)$ and $\lambda\in\Z^n$. This induces
a change of variables on the open orbit $\breve{X}_P$,
$$
 \iP\times\T^n \ni (x,\theta) \mapsto (l=Ax-\lambda,\vartheta={}^t A^{-1}\theta)
 \in (AP-\lambda)\times\T^n \subset (\R^+_0)^n\times\T^n
$$
such that $\omega=\sum \dd x_i\land\dd \theta_i = \sum
 \dd l_i\land \dd \vartheta_i$. 

Consider now the union
$$
 \iP_{v} := \{v\}\cup \bigcup_{\scriptsize{\begin{array}{@{}c@{}} F\ \rm{ face }
 \\ v\in\ol{F} \end{array}}} \breve{F}
$$
of the interior of all faces adjacent to $v$ and set $V_{v}:=\mu_P^{-1}(\iP_{v})
\subset X_P$. This is an open neighborhood of the fixed point $\mu_P^{-1}(v)$, and
it carries a smooth chart $V_{v} \to \C^n$ that glues to the chart on the open
orbit as
\begin{equation}
\label{tf}
 (A\iP-\lambda)\times\T^n \ni (l,\vartheta) \mapsto w_{v} = ( w_{l_j}= \e^{ y_{l_j}+\ii
 \vartheta_j} )_{j=1}^n \in \C^n ,
\end{equation}
where $y_{l_j}=\pd{g}{l_j}$.
We will call this ``the chart at $v$'' for short, dropping any reference to the choice
of ordering of the facets at $v$ to disburden the notation; usually we will then write
simply $w_j$ for the components of $w_v$. It is easy to see that a change of coordinates 
between two such
charts, at two vertices $v$ and $\widetilde{v}$, is holomorphic. 
The complex manifold, $W_P$, obtained by taking the vertex complex charts and the transition 
functions between them, 
does not depend on the symplectic potential. 
It will be convenient for us to distinguish between $X_P$ and $W_P$, 
noting that the $g$-dependent map $\chi_g: X_P\to W_P$ described locally by (\ref{tf}), 
introduces the $g$-dependent complex structure (\ref{cs}) on $X_P$, making $\chi_g$ a 
biholomorphism.

\subsection{The prequantum line bundle} In the same way, the holomorphic line bundle $L_P$ on $W_P$ determined canonically by the polytope $P$ \cite{Od}, 
induces, via the pull-back by $\chi_g$, an holomorphic structure on the (smooth) prequantum line bundle $L_\om$ on $X_P$, as will be described 
below:
\[
 \cdiag{ L_\om \ar[r] \ar[d] & L_P \ar[d] \\
 X_P \ar[r]^{\chi_g} & W_P \\
 }.
\]
Using the charts $V_v$ at the fixed points, one can describe $L_P$ by
\[
 L_P = \left( \coprod_{v} V_v\times\C \right)/\sim,
\]
where the equivalence relation $\sim$ is given by the transition functions
for the local trivializing holomorphic sections
${\mathbbm 1}_v(w_v)=(w_v,1)$ for $p\in V_v$,
\[
 {\mathbbm 1}_{v} = w_{\widetilde{v}}^{(\widetilde{A}A^{-1}\lambda-\widetilde{\lambda})}
  {\mathbbm 1}_{\tilde v},
\]
on intersections of the domains. Here, we use the data of the affine
changes of coordinates
\[
 l=\ell(x)=Ax-\lambda, \qquad \widetilde{l} = \widetilde{\ell}(x) = \widetilde{A}x-\widetilde{\lambda},
\]
associated to the vertices $v$ and $\widetilde{v}$ (and the order of the
facets there). 

Note that one also has a trivializing section on the open orbit, such that for a vertex 
$v\in P$, ${\mathbbm 1}_v=w_{v} \breve {\mathbbm 1}$  on $\breve W_P$. However, note that the sections 
${\mathbbm 1}_v$ extend to global holomorphic sections on $W_P$ while the extension of $\breve {\mathbbm 1}$ will, in general, 
be meromorphic and will be holomorphic iff $0\in P\cap \Z^n$. 
Sections in the standard basis 
$\{\sigma^m\}_{m\in P\cap\Z^n}\subset H^0(W_P,L_P)$ read $\sigma^m = \sigma^m_{\breve P} \breve{\mathbbm 1}
= \sigma_v^m {\mathbbm 1}_v$, with
\[
 \sigma_{\iP}^m(w) = w^m, \qquad
 \sigma_v^m(w_v) = w_v^{\ell(m)},
\]
in the respective domains.

The prequantum line bundle $L_\om$ on the symplectic manifold $X_P$ is, analogously,
defined by unitary local trivializing sections $\u1_v$ and transition functions
\be
\label{tfu}
 \u1_v = \e^{\ii\t(\widetilde{A}A^{-1}\lambda-\widetilde{\lambda})\tilde\vartheta}
\u1_{\tilde v}.
\ee
$L_\om$ is equipped with the compatible prequantum connection $\nabla$,
of curvature $-\ii\om$, defined by
\be
\label{conn2}
\nabla \u1_v= -\ii\t(x-v)\dd\theta\, \u1_v = -\ii\t l \,\dd\vartheta\, \u1_v,
\ee
where we use $\ell(v)=0$.

The bundle isomorphism relating $L_\om$ and $L_P$ is determined by 
\begin{eqnarray}
\nonumber
\u1_v =\e^{h_v\circ \mu_P} \chi_g^* \mathbbm{1}_v , \qquad 
\bu1 =\e^{h\circ \mu_P} \chi_g^*\breve{\mathbbm{1}},
\end{eqnarray}
where, for $m\in \Z^n$, $h_m(x)=(x-m)\pd{g}{x}-g(x)$ and $h$ is the function 
in (\ref{herm}) defining the inverse Legendre transform. 

In these unitary local trivializations, the sections $\chi_g^* \sigma^m$ read 
\[
\chi_g^*\sigma^m = \e^{-h_m\circ \mu_P}\e^{\ii \t m\theta} \bu1 =
\e^{-h_{m}\circ\mu_P }\e^{\ii\t \ell(m)\vartheta}\u1_v,
\]
where, after an affine change of coordinates $x\mapsto\ell(x)$ 
on the moment polytope, as above, we get $h_m(l) = \t(l-\ell(m))\pd{g}{l}-g(l)$.
Then, for all $\sigma\in H^0(W_P,L_P)$,  $\chi_g^* \sigma \in C^\infty (L_\om)$ is holomorphic, that is 
$$
\nabla_{\ol \xi} \chi_g^*\sigma =0, 
$$
for any holomorphic vector field $\xi$. That is, such sections are polarized with respect to the  
distribution of holomorphic vector fields on $X_P$ (see, for instance, \cite{Wo}).  

To treat the real polarization defined by the moment map, we will find it necessary to extend the operator of covariant differentiation
from smooth to distributional sections: we consider the injection of smooth
in distributional sections determined by Liouville measure,
\begin{equation}\label{iota}
 \begin{array}{rccl}
  \iota : & C^\infty(L_\omega\vert_U) & \longrightarrow & C^{-\infty}(L_\omega\vert_U) =
  \left( C^\infty_c(L_\omega^{-1}\vert_U) \right)' \\
 & s & \mapsto & \iota s (\phi) = \int\limits_U s \phi \frac{\om^n}{n!}
 \end{array} ,
\end{equation}
where $U\subset X_P$ is any open set. To extend the operator $\nabla_\xi$ on smooth sections to an operator we
denote $\nabla''_\xi$ on distributional sections we demand commutativity
of the diagram
\[
 \cdiag{
 C^{\infty}(L_\omega\vert_U) \ar@{^{(}->}[r]^\iota \ar[d]_{\nabla_{\xi}} &
 C^{-\infty}(L_\omega\vert_U) \ar[d]^{\nabla''_{\xi}} \\
 C^{\infty}(L_\omega\vert_U) \ar@{^{(}->}[r]^\iota &
 C^{-\infty}(L_\omega\vert_U) \\
} .
\]
To determine $\nabla''_\xi\si$ for a general distributional section $\si$
not of the form $\iota s$, we establish what its transpose is by integrating
the operator $\nabla_\xi$ by parts. This gives, for any smooth section $s\in
C^\infty(L_\omega\vert_U)$ and smooth test section $\phi\in C^\infty_c(L_\omega^{-1}\vert_U)$ ,
\[
 \left( \nabla''_\xi\iota s\right)(\phi) = \int\limits_U \left( \nabla_\xi s \right)
 \phi \frac{\om^n}{n!}  
  =  -\int\limits_U s \left( \mathrm{div}\xi \phi+ \nabla^{-1}_\xi\phi \right)
 \frac{\om^n}{n!} .
\]
Here we use the fact that given a connection $\nabla$ on $L_\omega$, the inverse line bundle $L_\omega^{-1}$ (defined by the inverse cocycle in a trivialization) comes naturally equipped with a connection
(defined by the negative of the connection one-forms); we will denote this connection by $\nabla^{-1}$.
 Therefore, $\nabla''_\xi$ is characterized by its transpose,
\[
 \nabla''_\xi \si (\phi) = \si(\t \nabla_\xi \phi), \quad \forall
 \phi \in C^\infty_c(L_\omega^{-1}\vert_U) ,
\]
where
\[
 \t \nabla_\xi \phi = - \left( \mathrm{div} \xi \phi + \nabla^{-1}_\xi \phi \right) .
\]

\begin{rem} The formulae for the definition of weak covariant constancy would become
more involved if we wrote them using the Hilbert space structure on sections of $L$
given by the Hermitean structure.
For example, we can extend the operator
of covariant differentiation by use of the (restriction of the) adjoint
of $\nabla_\xi$ as operator on the dense subspace $C^\infty_c(L_\omega\vert_U) \subset 
L^2(L_\omega\vert_U)$. Note that
\[
 \langle s, s' \rangle_{L^2} = (\iota s)(\ol{s'} h) , \quad
 \forall s,s' \in C^\infty_c(L_\omega\vert_U),
\]
where $h\in C^\infty\left( (L_\omega\tns\ol{L_\omega})^{-1} \right)$ is the Hermitean
structure on the line bundle $L_\omega$. This gives
\[
 \langle \nabla_\xi s, s' \rangle_{L^2} =  \langle s,\nabla^*_\xi s' \rangle_{L^2}
 \iff
(\iota s) \t \nabla_\xi (\ol{s'} h)
 = (\iota s) \left( \left(\ol{ \nabla_\xi^* s'}\right) h \right) ,
\]
or
\[
 \nabla^*_\xi s = \ol{\t\nabla_\xi\left(\ol{s} h\right)  h^{-1}} .
\]
\end{rem}

\subsection{The space of toric K\"ahler metrics}\label{sect_tangentcone} If we denote the set of toric K\"ahler metrics on $(X_P,\omega)$ by $\mathcal{M}_P$, it is parametrized by the convex set of functions $C_{g_P}^\infty(P)$. The space $\mathcal{M}_P$ carries a Riemannian metric $\gamma_{\mathcal{M}_P}$ introduced by Mabuchi, Semmes and Donaldson (see \cite{SZ} and references therein), whose geodesic segments are linear in terms of  the symplectic potential $\vp$,
\[
 s \mapsto g_P + \vp_0 + s\left(\vp_1-\vp_0 \right) .
\]
{}From this, it is clear that the tangent cone at infinity (introduced by Gromov, and which we think of as ``the space seen from infinity'', cf. \cite{Gr,JM})
\begin{equation}\label{tangent_cone_limit}
 T_\infty\mathcal{M}_P := \lim_{t\to\infty} \left( \mathcal{M}_P, \frac{1}{t}\gamma_{\mathcal{M}_P} \right)
\end{equation}
consists of all functions $\psi \in C_{g_P}^\infty(P)$ with non-negative definite Hessian on the whole interior of $P$, which is the necessary and sufficient condition for the geodesic ray
\begin{equation}\label{geodesicray}
 s \mapsto g_P + \vp + s \psi 
\end{equation}
to be defined for all $s \geq 0$. Denoting this set by $C^\infty_{\textrm{Hess} \geq 0}(P)$, we have therefore a natural identification
\begin{equation}\label{tangent_cone_identification}
 T_\infty \mathcal{M}_P \cong C^\infty_{\textrm{Hess} \geq 0}(P) .
\end{equation}
Actually, for technical reasons we will restrict mostly to the subset
\[
 T_\infty^+ \mathcal{M}_P :\cong C^\infty_{\textrm{Hess} > 0}(P)
\]
of strictly convex directions in $\mathcal{M}_P$, for the following reason: if we consider the family of Riemannian metrics on $X_P$ over $\mathcal{M}_P$,
\[
 \cdiag{ \mathcal{X}_P \cong \mathcal{M}_P \times X_P \ar[d] \\ \mathcal{M}_P }, \textrm{ where } \left(\mathcal{X}_P\right)_\vp := \left( X_P, \gamma_\vp \right) ,
\]
we can ``lift the limit'' (\ref{tangent_cone_limit}) to the geodesic families.
\begin{figure}
 \begin{centering}
 \includegraphics[width=240pt]{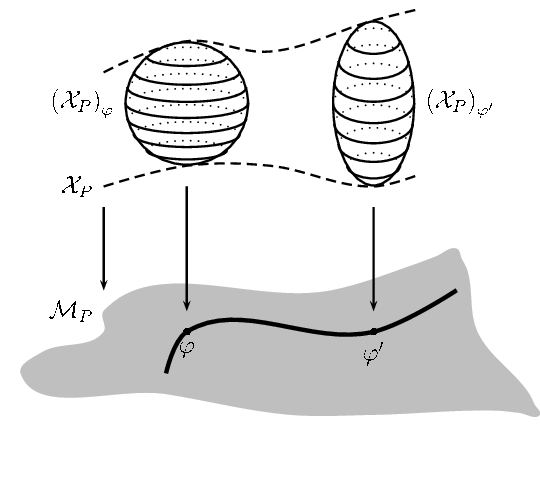}
 \end{centering}
 \caption{The family of toric K\"ahler metrics, schematically.}
 \label{fig16}
\end{figure}
Indeed, substituting the potential $g_P + \vp + s \psi$ in (\ref{cs}) and restricting to the diagonal $s=t$, we see that
\begin{equation}
 \frac{1}{s} \gamma_{\vp+s\psi} = \left(\begin{array}{cc}
 \frac{1}{s} {\rm Hess}_{x}\left( g_P+\vp \right) + {\rm Hess}_{x} \psi & 0\\
0 & \frac{1}{s} \left( {\rm Hess}_{x}\left( g_P+\vp \right) + s{\rm Hess}_{x} \psi \right)^{-1}\end{array}\right),
\end{equation}
that is,
\[
 \lim_{s\to\infty} \left( X_P, \frac{1}{s}\gamma_{\vp+s\psi} \right) = \left( P, {\rm Hess}_{x} \psi  \right) .
\]
It is in this sense that the family of toric K\"ahler metrics on $X_P$, when seen from infinity, collapses to a family of Hessian metrics on $P$ over $T_\infty\mathcal{M}_P$ that we denote (with a slight abuse of notation) by $T_\infty \mathcal{X}_P$,
\[
 \cdiag{ T_\infty \mathcal{X}_P \cong P \times T_\infty \mathcal{M}_P \ar[d] \\ T_\infty \mathcal{M}_P }, \textrm{ where } \left(T_\infty \mathcal{X}_P\right)_\vp := \left( P, {\rm Hess}_{x} \psi \right) .
\]
It is natural to turn the attention primarily to the limits which have a non-degenerate metric, that is, to $T_\infty^+ \mathcal{M}_P$.

In Section \ref{cpamoeba} we will study the geometry (and how much of it can be recovered from the limit) of generic divisors in $X_P$ along these lines.
\begin{figure}
 \begin{centering}
 \includegraphics[width=240pt]{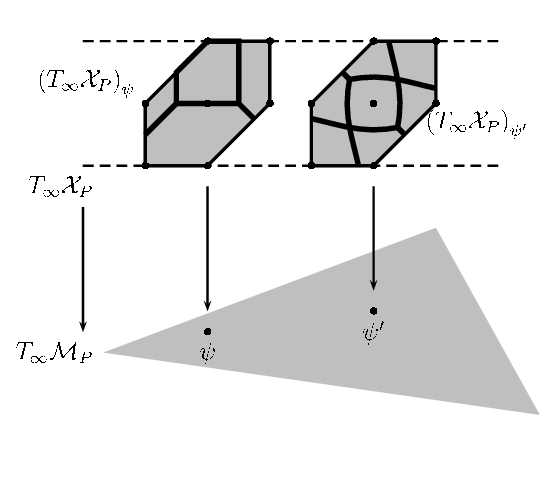}
 \end{centering}
 \caption{The family of Hessian metrics on a polytope with limit amoebas, schematically.}
 \label{fig15}
\end{figure}

\section{Quantization}

In this Section, we study the geometric quantization of toric symplectic manifolds for the family $\mathcal{M}_P$ of toric K\"ahler metrics, and their degeneration.

\subsection{Quantization in a real polarization}
\label{realpol}

We begin by describing the quantization obtained using the
definition outlined above for the singular real polarization $\P_\R$
defined by the moment map
\[
 \P_\R = \ker \dd \mu_P .
\]
This means that we consider the space of weakly covariant constant
distributional sections, $\Q_\R \subset (C^{\infty}(L_\om^{-1}))'$,
\[
 \Q_\R := \left\{\sigma\in (C^{\infty}(L_\om^{-1}))' |
 \forall\, W\subset X_P \textrm{ open, } \forall\, \xi\in C^{\infty}(\P_\R\vert_W),
 \, \nabla_{\xi}''(\sigma\vert_{W})=0 \right\} .
\]
In the present Section, ``covariantly constant'' is always understood
to mean ``covariantly constant with respect to the polarization $\P_\R$''.
We first give a result describing the local covariantly
constant sections. Note that it only depends on the local structure of the
polarization, and thus applies also in cases where globally one is not dealing
with toric varieties.

\begin{prop} \label{del} \begin{itemize}
 \item[(i)] Any covariantly constant section on a $\T^n$-invariant open set
$W\subset X_P$ is supported on the Bohr-Sommerfeld fibers in $W$.
\item[(ii)] The distribution
\[
 \de^{m}(\tau)=
 \int\limits _{\mu_P^{-1}(m)}\e^{\ii \t\ell(m) \vartheta}\tau_{v}, 
 \quad \forall\tau\in C^{\infty}_c(L_\om^{-1}\vert_W),
\]
is covariantly constant.
\item[(iii)] The sections $\de^m, \ m\in\mu_P(W)\cap\Z^n$ span the space
of covariantly constant sections on $W$.
\end{itemize}
\end{prop}

\begin{proof} Note that all the statements are local in nature. Using the action
of $Sl(n,\Z)$ on polytopes, we can reduce without loss of generality an arbitrary
Bohr-Sommerfeld fiber $\mu_P^{-1}(m)$ with $m\in P\cap\Z^n$ being contained in a
codimension $k$ (and no codimension $k+1$) face of $P$ to the form
\[
 m = (\underbrace{0,\dots,0}_k,\widetilde{m}) \quad \textrm{ with } \widetilde{m}_j > 0 \quad
 \forall j=k+1,\dots,n .
\]
Furthermore, we can cover all of $X_P$ by charts $\widetilde{W}$ of special
neighborhoods of $\mu_P^{-1}(m)$ of the form
\[
 \widetilde{W} := B_\varepsilon(0) \times
 \left(\widetilde{m}+]-\varepsilon,\varepsilon[^{n-k}\right)
 \times \T^{n-k}
\]
with $0 < \varepsilon < 1$ and $B_\varepsilon(0)\subset\R^{2k}$ being the ball
of radius $\varepsilon$ in $\R^{2k}$. This chart is glued onto the open orbit via the map,
\begin{eqnarray*}
& \breve{P}\times\T^n \ni (x_1,\dots,x_n,\th_1,\dots,\th_n) \mapsto \\
& \left( u_1 = \sqrt{x_1}\cos\th_1, v_1 = \sqrt{x_1}\sin\th_1,\dots, u_k = \sqrt{x_k}\cos\th_k,
 v_k = \sqrt{x_k}\sin\th_k, \right. \\
& \left. x_{k+1},\dots,x_n,\th_{k+1},\dots,\th_n \right) \in \widetilde{W}
\end{eqnarray*}
and therefore
induces the standard symplectic form in the coordinates $u,v,x,\th$
\[
 \om = \sum_{j=1}^k \dd u_j\land \dd v_j + \sum_{j=k+1}^n \dd x_j \land \dd \th_ j .
\]
We then trivialize the line bundle with connection $L_\om$
over $\widetilde{W}$ by setting
\[
 \nabla = \dd - \ii\left(\frac{1}{2}(\t u\dd v-\t v\dd u)+\t x\dd \th\right).
\]
Therefore, we are reduced to studying the equations of covariant constancy
on the space of (usual) distributions $C^{-\infty}(\widetilde{W})$ on
$\widetilde{W}$, using arbitrary test functions in $C^\infty_c(\widetilde{W})$
and vector fields $\xi\in C^\infty(\P_\R\vert_{\widetilde{W}})$. These can be
written as
\begin{equation} \label{xivf}
 \xi_{(u,v,x,\th)} = \sum_{j=1}^k \al_j
 \left(u_j\pd{}{v_j}-v_j\pd{}{u_j}\right) +
 \sum_{j=k+1}^n \beta_j\pd{}{\th_j} ,
\end{equation}
where $\al_j,\beta_j$ are smooth functions on $\widetilde{W}$.

 (i) First, note that if a distribution $\si$ is covariantly constant, it
is unchanged by parallel transport around a loop in $\T^n$, while on the other
hand this parallel transport results in multiplication of all test sections by
the holonomies of the respective leaves. Explicitly, for a loop specified by a
vector $a\in\Z^n$, any test function $\tau\in C^\infty_c(\widetilde{W})$ is
multiplied by the smooth function $\e^{2\pii \t a\, \mu_P(u,v,x,\th) }$. Therefore,
\[
 \si = \e^{2\pii \t a \, \mu_P } \si, \quad \forall a \in \Z^n,
\]
and $\si$ must be supported in the set where all the exponentials equal $1$,
i.e. on the Bohr-Sommerfeld fibers, where $\mu_P$ takes integer values.

(ii) In the chart on $\widetilde{W}$, any test function $\tau$ can be written
as a Fourier series
\begin{equation}\label{fwf}
 \tau(u,v,x,\th) = \sum_{b\in\Z^{n-k}} \widehat{\tau}(u,v,x,b) \e^{\ii \t b\, \th} ,
\end{equation}
with coefficients that are smooth and compactly supported in the other variables,
\[
 \forall b\in\Z^{n-k}: \quad
 \widehat{\tau}(.,b) \in C^\infty_c\left(
 B_\varepsilon(0)\times\left(\widetilde{m}+]-\varepsilon,\varepsilon[^{n-k}\right)
 \right) .
\]
The local representative of the distributional section $\de^m$ is calculated as
\begin{eqnarray*}
 \de^m(\tau) & = &  \int\limits _{\mu_P^{-1}(m)}\e^{\ii \t\widetilde{m} \th}\tau \ = \\
 & = & \int\limits_{\T^{n-k}} \e^{\ii \t\widetilde{m} \th} \tau(u=0,v=0,x=\widetilde{m},\th)
 \dd \th \ = \\
 & = & \widehat{\tau}(u=0,v=0,x=\widetilde{m},b=-\widetilde{m}) .
\end{eqnarray*}

Differentiating an arbitrary test function $\tau$ along any vector field
\[
 \xi_{(u,v,x,\th)} = \sum_{j=1}^k \al_j
 \left(u_j\pd{}{v_j}-v_j\pd{}{u_j}\right) +
 \sum_{j=k+1}^n \beta_j\pd{}{\th_j} ,
\]
with constant coefficients $\al_j,\beta_j$ in the polarization $\P$
(which is generated by such vector fields over $C^\infty(W)$) gives
\begin{eqnarray}
 \t\nabla_\xi \tau & = & \dd\tau \xi +
 \ii\left(\frac{1}{2}(\t u\dd v-\t v\dd u)+\t x\dd \th\right)\xi \tau \ = \nn\\
 & = & \sum_{b\in\Z^{n-k}} \left(
 \sum_{j=1}^k \al_j \left(u_j\pd{\widehat{\tau}}{v_j}-v_j\pd{\widehat{\tau}}{u_j}
 +\frac{\ii}{2}(u_j^2+v_j^2)\widehat{\tau}\right) + \right. \nn\\
 & & \quad \left. +\ii \sum_{j=k+1}^n \beta_j (x_j+b_j)\widehat{\tau} \right)
 \e^{\ii \t b\, \th} , \label{dertst}
\end{eqnarray}
whence for arbitrary $\tau$ and any such $\xi$
\begin{eqnarray*}
 \de^m( \t \nabla_\xi \tau ) & = &
  \left( \sum_{j=1}^k \al_j \left(u_j\pd{\widehat{\tau}}{v_j}-v_j\pd{\widehat{\tau}}{u_j}
 +\frac{\ii}{2}(u_j^2+v_j^2)\widehat{\tau}\right) + \right. \\
 & & \quad \left. + \ii \sum_{j=k+1}^n \beta_j (x_j+b_j)\widehat{\tau}
 \right)_{(u=0,v=0,x=\widetilde{m},b=-\widetilde{m})} \ = \ 0 ,
\end{eqnarray*}
so that $\de^m$ is covariantly constant.

(iii) Consider an arbitrary distribution $\si\in C^{-\infty}(\widetilde{W})$.
Using Fourier expansion of test functions along the non-degenerating directions of
the polarization on $\widetilde{W}$, as in item (ii), to $\si$ we associate a family
of distributions $\widehat{\si}_b$ on $B_\varepsilon(0)\times
\left(\widetilde{m}+]-\varepsilon,\varepsilon[^{n-k}\right)$ by setting
\[
 \widehat{\si}_b(\psi) := \si\left(\psi(u,v,x)\e^{\ii \t b \th}\right) , \quad
\forall \psi \in  C^\infty_c\left(
 B_\varepsilon(0)\times\left(\widetilde{m}+]-\varepsilon,\varepsilon[^{n-k}\right)
 \right) ,
\]
so that
\[
 \si(\tau) = \sum_{b\in\Z^{n-k}} \widehat{\si}_b(\widehat{\tau}(.,b)) .
\]
It is clear that the map $\si\mapsto(\widehat{\si}_b)_{b\in\Z^{n-k}}$ is injective,
so we need to show that the condition of covariant constancy
\[
 \si( \t\nabla_\xi \tau ) = 0, \quad \forall \tau\in C^\infty_c(\widetilde{W}),
 \xi\in C^\infty(\P_\R\vert_{\widetilde{W}})
\]
implies that
\[
 \widehat{\si}_b = \left\{ \begin{array}{ll} 0 & \textrm{ if } b\neq -\widetilde{m}, \\
 \la \de(u,v,x-\widetilde{m}) & \textrm{ if } b = -\widetilde{m}\textrm{, where }
 \la\in\C. \end{array} \right.
\]
{}From (i) we know that $\widehat{\si}_b$ has support at the point
$(u,v,x)=(0,0,\widetilde{m})$ for each $b$, therefore it must be of the form
(see, for instance, \cite{Tr} Chapter 24; here $j,k,l$ are multi-indices)
\[
 \widehat{\si}_b = \sum_{\mathrm{finite}} \ga_{jkl}^b
\pd{{}^{|j|+|k|+|l|}}{u^j\partial v^k \partial x^l} \de(u,v,x-\widetilde{m}) .
\]
Using for example a vector field $\xi$ with $\al_j=\beta_j\equiv 1$ and
explicit test functions $\tau$ that are polynomial near the point
$(u,v,x)=(0,0,\widetilde{m})$, it is easy to give examples that show
that if $\ga_{jkl}^b\neq 0$ and either $b\neq-\widetilde{m}$ or $|j|+|k|+|l|>0$, then
there exist $\xi$ and $\tau$ such that
\[
 \si( \t\nabla_\xi \tau ) \neq 0,
\]
thus producing a contradiction.
\end{proof}

Immediately from this proposition follows the

\begin{proofb}{\ref{th1b}} Since the open neighborhoods considered in the
Proposition are $\T^n$-invariant and the covariantly constant sections are
supported on closed subsets, each of them can evidently be extended by $0$
to give a global covariantly constant section.
\end{proofb}

\subsection{The degenerate limit of K\"ahler polarizations}
\label{largepol}

As mentioned above, we will turn our attention to a 
family of compatible complex structures
determined by the symplectic potentials $g_{s}=g_{P}+\varphi+s\psi$, being interested in the limit
of holomorphic polarizations, 
in the sense of geometric quantization, 
and subsequently in the convergence of monomial sections
to the delta distributions just described.

In the vertex coordinate charts $V_{v}$ described above, 
the holomorphic polarizations
are given by
\[
 \P_\C^{s}=\spanC\{\frac{\partial}{\partial w_{j}^{s}}:j=1,\dots,n\}.
\]
Let $\P_{\R}$ stand for the vertical polarization, that is
\[
 \P_{\R}:=\ker\dd\mu_P,
\]
 which is real and singular above the boundary $\partial P$. Let now 
$\P^{\infty} := \lim_{s\to\infty}\P_\C^{s}$ where the limit is taken in the positive 
Lagrangian Grassmannian of the complexified tangent space at each point in $X_P$.

\begin{lem}
\label{openo}
On the open orbit $\breve X_P$, $\P^{\infty}=\P_{\R}.$
\end{lem}
\begin{proof}
By direct calculation,
\[
 \frac{\partial}{\partial y_{l_{j}}^{s}}=\sum_{k}(G_{s}^{-1})_{jk}\frac{\partial}{\partial l_{k}}
\]
 where
\[
 (G_s)_{jk}=\Hess g_{s}
\]
 is the Hessian of $g_{s}$. Since $G_{s} > s \Hess \psi >0$,
$(G_{s}^{-1})_{jk}\to 0$ and
\[
 \spanC\{\frac{\partial}{\partial w_{j}^{s}}\} = \spanC \{
 \frac{\partial}{\partial y_{l_{j}}^{s}}-\ii\frac{\partial}{\partial\vartheta_{j}}
 \} \to \spanC\{\frac{\partial}{\partial\vartheta_{j}}\}.
\]
\end{proof}
At the points of $X_P$ that do not lie in the open orbit, 
the holomorphic polarization ``in the degenerate angular directions'' is 
independent of $g_s$, in the following sense:

\begin{lem}
Consider two charts around a fixed point $v\in P$, $w_v,\widetilde{w}_v:V_{v}\to\C^{n}$,
specified by symplectic potentials $g\neq\widetilde{g}$.

Whenever $w_{j}=0$, also $\widetilde{w}_{j}=0$, and
at these points \[
\C\cdot\frac{\partial}{\partial w_j}=\C\cdot
\frac{\partial}{\partial\widetilde{w}_{j}},\quad j=1,\dots,n.\]
 
\end{lem}
\begin{proof}
According to the description of the charts, \[
w_{j}=\widetilde{w}_j f,\]
 where $f$ is a real-valued function, smooth in $P$, that factorizes through
$\mu_P$, that is through $|\widetilde{w}_1|,\dots,|\widetilde{w}_{n}|$.
Therefore, 
\[
\dd w_j=\sum_{k}\big[(\de_{j,k}f+\widetilde{w}_j\frac{\partial f}{\partial\widetilde{w}_k})\dd\widetilde{w}_{k}
+\widetilde{w}_{j}\frac{\partial f}{\partial\overline{\widetilde{w}}_{k}}\dd\overline{\widetilde{w}}_{k}\big]\]
 and at point with $\widetilde{w}_{j}=0$ one finds, in fact, \[
\C\cdot\frac{\partial}{\partial w_{j}}=\C\cdot\frac{\partial}{\partial\widetilde{w}_{j}}.\]
 
\end{proof}
The two lemmata together give

\begin{thm} \label{oanterior}
In any of the charts $w_v$, the limit polarization $\P^{\infty}$ is
\[
\P^{\infty}=\P_{\R}\oplus\spanC\{\frac{\partial}{\partial w_j}:w_{j}=0\}\]
 \end{thm}

\begin{proof}
It suffices to show the convergence \[
\spanC\{\frac{\partial}{\partial w_{k}^{s}}\}\to\spanC\{\frac{\partial}{\partial\theta_{k}}\}\]
 whenever $w_{k}\neq0$, which really is a small modification
of Lemma \ref{openo}. For any face $F$ in the coordinate neighborhood,
we write abusively $j\in F$ if $w_j=0$ along $F$. For
any such affine subspace we have then \begin{eqnarray}
\nn\P_\C^{s} & = & \spanC\{\frac{\partial}{\partial w_{j}^{s}}:j\in F\}\oplus\spanC\{\frac{\partial}{\partial y_{l_{k}}^{s}}-\ii\frac{\partial}{\partial\vartheta_{k}}:k\notin F\}=\\
\nn & = & \spanC\{\frac{\partial}{\partial w_{j}^{s}}:j\in F\}\oplus\spanC\{\sum_{k'\notin F}((G_s)_{F})_{kk'}^{-1}\frac{\partial}{\partial l_{k'}}-\ii\frac{\partial}{\partial\vartheta_{k}}:k\notin F\}\end{eqnarray}
 where $(G_s)_{F}$ is the minor of $G_{s}$ specified by the variables
that are unrestricted along $F$, which is well-defined and equals the Hessian of
the restriction of $g_s$ there. Since this tends to infinity, its inverse
goes to zero and the statement follows.
\end{proof}

This results in

\begin{proofb}{\ref{th1}} Given Theorem \ref{oanterior}, we are reduced
to proving that
\[
 C^\infty\left(\P_{\R}\oplus\spanC\{\frac{\partial}{\partial w_j}:w_{j}=0\}\right)
 = C^\infty(\P_\R).
\]
This is clear since any
smooth complexified vector field $\xi\in C^{\infty}(T^{\C}X)$
that restricts to a section of $\P_{\R}$ on an open dense subset
must satisfy $\overline{\xi}=\xi$ throughout. Such a vector field
cannot have components along the directions of
$\spanC\{\frac{\partial}{\partial w_{j}}:w_{j}=0\}$. 
\end{proofb}

\begin{rem} \label{defquant}
Note that the Cauchy-Riemann conditions hold for distributions
(see e.g. \cite{Gun} for the case $n=1$, or Lemma 2 in \cite{KY}),
that is, for any complex polarization $\P_\C$, considering the
intersection of the kernels of $\nabla''_{\pd{}{\ol{z}_j}}$
gives exactly the space
\[
 \bigcap_{\pd{}{\ol{z}_j} \in \ol{\P_{\C}}} \ker
 \nabla''_{\pd{}{\ol{z}_j}} = \iota H^{0}(X_P(\P_\C),L_\om(\P_\C))
 \subset (C^{\infty}(L_\om^{-1}))'
\]
 of holomorphic sections (viewed as distributions). Thus one can
view the 1-parameter family of quantizations associated to $g_{s}$
and the real quantization on equal footing, embedded in the space
of distributional sections, $\Q_s\subset (C^{\infty}(L_\om^{-1}))'$,
\[
 \Q_{s} := \left\{\sigma\in (C^{\infty}(L_\om^{-1}))'|
 \forall\, W\subset X_P \textrm{ open, } \forall\, \xi\in C^{\infty}(\ol{\P^{s}}\vert_W),
 \nabla_{\xi}''(\sigma\vert_{W})=0 \right\},
\]
for $s\in[0,\infty]$, where $\P^\infty := \P_\R$, and $\Q_\infty = \Q_\R$
in the previous notation. In the next Section, we will see that the
convergence of polarizations proved here translates eventually into
a continuous variation of the subspace $\Q_{s}$ in the space of
distributional sections as $s\to\infty$.
\end{rem}

\begin{rem}
Notice that Theorem \ref{th1} in this and Proposition \ref{del} in the previous
Section are also valid for non-compact $P$, with some obvious changes such as
taking test sections with compact support.
\end{rem}

\subsection{Degeneration of holomorphic sections and BS fibers}
\label{lcsl}

Here, we use convexity to show that, as $s\to\infty$, the holomorphic
sections converge, when normalized properly, to the distributional sections
$\delta^m$, described
in Proposition \ref{del}, and that are supported along the Bohr-Sommerfeld fibers of $\mu_P$ and
covariantly constant along the real polarization. 

First, we show an elementary lemma on certain Dirac sequences
associated with the convex function $\psi$ that will permit us to prove
Theorem \ref{th2}.

\begin{lem} 
\label{convex}
For any $\psi$ strictly convex in a neighborhood of
the moment polytope $P$ and any $m\in P\cap\Z^n$, the function
\[
 P \ni x \mapsto f_m(x) := \t(x-m)\pd{\psi}{x}-\psi(x) \in \R
\]
has a unique minimum at $x=m$ and
\[
 \lim_{s\to\infty} \frac{\e^{-sf_m}}{\|\e^{-sf_m}\|_1} \to \delta(x-m),
\]
in the sense of distributions.
\end{lem}
\begin{proof}
{}For $x$ in a convex neighborhood of $P$, using
$f_m(m)=-\psi(m)$ and $\nabla f_m(x)= \t (x-m) \Hess_x\psi$, we 
have
\begin{eqnarray}
\nn
f_m(x)&=& f_m(m) +\int_0^1 \frac{d}{dt}f_m(m+t(x-m)) \dd t =\\
\nn 
&=& - \psi(m) + \int_0^1 t \  \t(x-m) \left( \Hess_{m+t (x-m)}\psi \right) (x-m) \,\dd t.
\end{eqnarray}
Since $\psi$ is strictly convex, with
\[
 \Hess_x\psi > 2c I ,
\] 
for some positive $c$ and all $x$ in a neighborhood of $P$, it follows that 
$$
f_m(x)\geq -\psi(m) + c \|x-m\|^2.
$$
Obviously, $m$ is the unique absolute minimum of $f_m$ in $P$. For the
last assertion, we show that the functions 
\[
 \zeta_{s} := \frac{\e^{-sf_m}}{\|\e^{-sf_m}\|_1}
\]
form a Dirac sequence. (We actually show convergence as measures on $P$.)
It is clear from the definition that $\zeta_{s} >0$ and
$\|\zeta_{s}\|_{1}=1$, so it remains to show that the norms concentrate
around the minimum, that is, given any $\ep,\ep'>0$ we have to find a
$s_{0}$ such that
\[
 \forall s\geq s_{0}:\quad\int\limits _{B_{\ep}(m)}\zeta_{s}(x)\dd x\geq1-\ep'.
\]

Let $r_0>0$ be sufficiently small and let $2\alpha$ be the maximum of $\Hess_x\psi$ in 
$B_{r_0}(m)$. Observe that 
$$
\|\e^{-sf_m}\|_{1} = \int_P \e^{-sf_m(x)} \dd x \geq \int_{B_r(m)}\e^{-sf_m(x)} \dd x
\geq  d_n r^n \e^{s\psi(m)-s\alpha r^2},
$$ 
for any $r>0$ such that $r_0>r$, and where $d_n r^n= \Vol (B_r(m))$. On the other hand,
\begin{equation}
\label{estimate}
\int\limits _{P\setminus B_{\ep}(m)}\e^{-sf_m(x)}\dd x 
\leq  
\int\limits _{P\setminus B_{\ep}(m)}\e^{s\psi(m)-sc\|x-m\|^2}\dd x 
\leq \Vol(P) \e^{s\psi(m)} \e^{-sc\varepsilon^2}. 
\end{equation}
Therefore,
\begin{equation}
\nn
\int\limits _{P\setminus B_{\ep}(m)}\zeta_{s}(x)\dd x \leq
\frac{\Vol(P)\e^{-sc\varepsilon^2+s\alpha r^2}}{d_nr^n}.
\end{equation}
Choosing $r$ sufficiently small, the right hand side goes to zero as $s\to \infty$ and 
the result follows.
\end{proof}

Let now $\{\sigma^m_s\}_{m\in P\cap\Z^n}$ be the basis of holomorphic sections of $L_\om$, with respect to the holomorphic 
structure induced from the map $\chi_{g_s}$ in Section \ref{prelim}, that is $\sigma^m_s = \chi_{g_s}^*(\sigma^m)$, for $\sigma^m\in H^0(W_P,L_P)$.
We then have

\begin{proofb}{\ref{th2}} 
Using a partition of unity $\{\rho_{v}\}$ subordinated to the covering by the vertex coordinate charts 
$\{\iP_{v}\}$,
the result can be checked chart by chart. Let $\tau\in C^\infty(L_\om^{-1})$
be a test section and let
\begin{eqnarray}
\nn h_{m}^{s}(x)=\t(x-m)y-g_{s} & = & \big[\t(x-m)(\pd{g_{P}}{x}+\pd{\varphi}{x})-g_{P}-\varphi\big]+s\big[\t(x-m)\pd{\psi}{x}-\psi\big]=\\
\nn & = & h_{m}^{0}(x)+sf_m(x),
\end{eqnarray}
with $f_m$ as in Lemma \ref{convex}. Then
\begin{eqnarray*}
 (\iota(\xi_s^m))(\tau) & = & \frac{1}{\|\si_{s}^{m}\|_{1}}
 \sum_{v} \int\limits _{V_{v}} \rho_{v}\circ\mu_P(w_v) \e^{-h_m^{s}\circ\mu_P(w_v)}
 \e^{\ii \ell(m) \vartheta} \tau_{v}(w_v)\om^{n} = \\
 & = & \frac{1}{\|\si_{s}^{m}\|_{1}}
 \sum_{v} \int\limits _{\iP_{v}} \rho_{v}(x) \e^{-h_m^{s}(x)}
 \big( \int\limits_{\mu_P^{-1}(x)} \e^{2\pi\ii \ell(m) u}
 \tau_{v}(\e^{\pd{g_s}{l}(\ell(x))+2\pi \ii u})
 \dd u \big)
\dd x = \\
 & = &  \frac{1}{\|\si_{s}^{m}\|_{1}}
 \int\limits _{P} \e^{-h_m^{s}(x)} \widehat{\tau}(x,-m) \dd x,
\end{eqnarray*}
where $\widehat{\tau}$ is the fiberwise Fourier transform from equation (\ref{fwf}).

Now the $L^{1}$-norm in question calculates as \[
\|\sigma^m_s\|_{L^1}=\int\limits _{X_{P}}\e^{-h^{s}_m\circ\mu_P} \om^{n}=(2\pi)^n \int\limits _{P}\e^{-h_{m}^{s}}\dd x.\]
According to Lemma \ref{convex}, \[
\frac{\|\e^{-h_{m}^{0}-sf_m}\|_{1}}{\|\e^{-sf_m}\|_{1}}=\int\limits _{P} 
\frac{\e^{-sf_m}}{\|\e^{-sf_m}\|_{1}}\e^{-h_{m}^{0}}\dd x\to\e^{-h_{m}^{0}(m)}\textrm{ as }s\to\infty,\]
 and therefore
\[
 \iota(\xi_{s}^{m})(\tau) = \int\limits _{P}
 \frac{\e^{-h_{m}^{0}+sf_m}}{\|\e^{-h_{m}^{0}+sf_m}\|_{1}}
 \widehat{\tau}(.,-m)\dd x \to \widehat{\tau}(m,-m) = \de^m(\tau)
\]
which finishes the proof.
\end{proofb}

\begin{cor}
The results of Lemma \ref{convex} and Theorem \ref{th2} are valid 
for non-compact toric manifolds if one assumes uniform strict convexity of $\psi$. 
\end{cor}
\begin{proof}
In Lemma \ref{convex}, 
the estimate for $f_m(x)$ remains valid 
for any $x\in P$, so that the function $e^{-sf_m}$ will be integrable even if $P$ 
is not compact. As for the second part of the proof of Lemma \ref{convex},
instead of (\ref{estimate}) one can use
\begin{eqnarray}
\nn \int_{P\setminus B_\varepsilon(m)} \e^{-sf_m} \dd x & \leq & 
d_n \int_\varepsilon^{+\infty} \e^{s\psi(m)}\e^{-scr^2} r^{n-1} \dd r \leq
M d_n \int_\varepsilon^{+\infty} \e^{s\psi(m)}\e^{-s\frac{c}{2}r^2} \dd r \leq \\ \nn
& \leq & M d_n \e^{s\psi(m)} \e^{-s\frac{c}{2}\varepsilon^2}\int_0^{+\infty} \e^{-s\frac{c}{2}u^2} \dd u
\leq M' \e^{s\psi(m)} \e^{-s\frac{c}{2}\varepsilon^2}, 
\end{eqnarray}
for appropriate constants $M,M'>0$, where $M,M'$ depend on $s$ but are 
bounded from above as 
$s\to \infty$, so that the assertion follows.
\end{proof}

\section{Compact tropical amoebas}
\label{cpamoeba}

In this section, we undertake a detailed study of the behavior of the
compact amoebas in $P$ associated to the family of symplectic potentials in (\ref{gg})
\[
 g_s=g_P+\varphi+s\psi ,
\]
which define the complex structure $J_s$ on $X_P$, 
and of their relation to the $Log_t$ amoebas in $\R^n$ \cite{GKZ,Mi,FPT,R}.

Let $\breve{Z}_{s}\subset(\C^{*})^{n}$ be the complex hypersurface
defined by the Laurent polynomial
\[
\breve{Z}_{s}=\{ w\in(\C^{*})^{n}:\sum_{m\in P\cap\Z^{n}}a_{m} \e^{-sv(m)}w^{m}=0\},
\]
where $a_m\in \C^*, v(m)\in \R$.
 One natural thing to do in order to obtain the large K\"ahler structure
limit, consists in introducing the complex structure on $(\C^{*})^{n}$
defined by the complex coordinates $w=e^{z_{s}}$ where $z_{s}=sy+i\theta$,
and taking the $s\to+\infty$ limit. Then, the map $w\mapsto y$ coincides
with the $Log_{t}$ map for $s=\log t.$ However, this deformation
of the complex structure, which is well defined for the open dense
orbit $(\C^{*})^{n}\subset X_{P}$ never extends to any (even partial)
toric compactification of $(\C^{*})^{n}$. Indeed, that would correspond
to rescaling the original symplectic potential by $s$, which is incompatible
with the correct behavior at the boundary of the polytope found by
Guillemin and Abreu. 

As we will describe below, for deformations in the direction of quadratic $\psi$
in (\ref{gg}), in the limit we obtain the $Log_{t}$ map amoeba
intersected with the polytope $P$. The significative difference is that our
limiting tropical amoebas are now compact and live inside $P$. For
more general $\psi$, they live in the compact image of $P$ by the
Legendre transform $\L_{\psi}$ in (\ref{ltr}) and are determined by the locus of
non-dif\-fer\-en\-tia\-bi\-li\-ty of a piecewise linear function, namely as
the tropical amoeba of \cite{GKZ,Mi},
\[
 \Atrop := C^0\!-{\rm loc} \big( u\mapsto \max_{m\in P\cap\Z^n}
 \{ \t m u-v(m) \} \big) .
\]

\subsection{Limit versus tropical amoebas}

We are interested in the $\mu_P$-image of the family of (complex) hypersurfaces
\[
 Y_s := \{ p \in X_P: \sum_{m\in P\cap\Z^n} a_m \e^{-s v(m)} \sigma_s^m(p) = 0 \}
 \subset (X_P,J_s)
\]
where $a_m\in\C^*$ and $v(m)\in\R$ are parameters and $\sigma_s^m
 \in H^0\left((X_P,J_s),\chi_{g_s}^* L_P\right)$
is the canonical basis of holomorphic sections of the line bundle $\chi_{g_s}^* L_P$
associated with the polytope $P$ and the symplectic potential (\ref{gg}), introduced
in Section \ref{prelim}. We call 
the image $\mu_P(Y_s)\subset P$ the compact amoeba of $Y_s$ in $P$.

\begin{dfn} The limit amoeba $\Alim$ is the subset
\[
 \Alim := \lim_{s\to\infty} \mu_P(Y_s)
\]
of the moment polytope $P$, where the limit is to be understood
as the Hausdorff limit of closed subsets of $P$.
\end{dfn}

The existence of this limit is shown in the proof of Theorem \ref{amoeba_thm} below. 
We will
relate this amoeba to the tropical amoeba of \cite{GKZ,Mi} using a Legendre transform $\breve{\chi}_s$ that is the restriction
of the map $\chi_{g_s}$ described in Section \ref{prelim} to the open orbit $\breve{X}_P$:
\[
 \cdiag{
 Y_s \ar@{}[d]|\bigcap & \breve{Y}_s \ar[r]^{\cong} \ar@{_{(}->}[l] %
 \ar@{}[d]|\bigcap & \breve{Z}_s \ar@{}[d]|\bigcap \ar@{^{(}->}[r] & %
 Z_s \ar@{}[d]|\bigcap \\
 (X_P,\om,J_s,G_s) \ar[d]_{\mu_P} & \breve{X}_P \ar[r]_-{\breve{\chi}_s}^-{\cong} \ar[l] %
                   \ar[d]_{\mu_P} \ar@{_{(}->}[l]                      & %
                   (\C^*)^n \ar[d]^{Log_t} \ar@{^{(}->}[r]  & (W_P,J) \\
 P              & \iP \ar[r]_{\kappa_s}^{\cong} \ar@{_{(}->}[l]            & %
 \R^n &
}
\]
where $\kappa_s$ is the family of rescaled Legendre transforms
\[
 \iP \ni x \mapsto \kappa_{s}(x) := \frac{1}{s} \L_{(g_{P}+\varphi)+s\psi} =
 \pd{\psi}{x} + \frac 1s \pd{(g_P+\varphi)}{x} \in \R^n .
\]
For any $s>0$, this is a diffeomorphism $\iP\to\R^{n}$.

Let $\A_s:=Log_t(\breve{Z}_s)$ be the amoeba of \cite{GKZ,Mi}. Recall that
$\A_s\to\Atrop$ in the Hausdorff topology \cite{Mi,R}.

\begin{prop} The family of rescaled Legendre transforms $\kappa_s$ satisfies
\[ \kappa_s \circ \mu_P(\breve{Y}_s) = \A_s
\]
\end{prop}
\begin{proof} Under the trivialization of $L_P$ determined by $g_s$ on
the open orbit $\breve{X}_P$, the sections $\sigma^m_s(x,\theta)$ correspond to
polynomial sections $w^m$, where
\[
  w = \e^{\pd{(g_P+\varphi)}{x}+ s \pd{\psi}{x} +
   \ii \theta} .
\]
Combining this with the $Log_t$-map for $t=\e^s$ gives precisely
\[
 Log_t w = \pd{\psi}{x}+\frac 1s \pd{(g_P+\varphi)}{x} = \kappa_s(x) .
\]
\end{proof}

\begin{rem} Note that since $Y_s$ is defined as the zero locus of a
  global section, one has $Y_s = \overline{\breve{Y}_s}$
and, in particular,
$\mu_P(Y_s) = \overline{\kappa_s^{-1}\A_s}$.
On each face $F$ of the moment polytope $P$, this will consist exactly
of the amoeba defined by the sum of monomials corresponding to integer
points in $F$, cf. \cite{Mi}.
\end{rem}

The family of inverse maps $\kappa_s^{-1}$ will permit us
to capture information not only concerning the open orbit but also 
the loci of compactification of $X_P$.

\begin{lem} \label{k_conv} For any compact subset $C \subset \iP$ and any 
$\psi$ strictly convex on a neighborhood of $P$,
\[
 \kappa_s \to \L_\psi \textrm{ pointwise on } \iP \textrm{ and uniformly
 on } C
\]
and
\[
 \kappa_s^{-1} \to \L_\psi^{-1} \textrm{ uniformly on } \L_\psi C .
\]
\end{lem}
\begin{proof} Since $\pd{(g_P+\varphi)}{x}$ is a smooth function on $\iP$,
$\kappa_s \to \L_\psi$ pointwise on $\iP$ and uniformly on compact
subsets $C\subset \iP$. Furthermore,
\begin{equation} \label{unif_inj}
 \| \kappa_s(x)-\kappa_s(x') \| \geq c \| x-x' \|, \qquad \forall x,x' \in \iP,
\end{equation}
with a constant $c>0$ uniform in $s$, since the derivative
\[
 \pd{\kappa_s}{x} = \Hess_x \psi + \frac 1s \Hess_x (g_P+\varphi)
 > \Hess_x \psi > 0
\]
is (uniformly) positive definite. Therefore, the family of 
inverse mappings $\kappa_s^{-1}$ is uniformly Lipschitz (on $\R^n$).
Thus the pointwise convergence $\kappa_s^{-1}\to\L_\psi^{-1}$ on
$\L_\psi \iP$ is uniform on any compact $\L_\psi C$.
\end{proof}

Before proving the main theorem, we recall some facts about convex sets in $\R^n$ and also 
show two auxiliary lemmata on the
behavior of the gradient of any toric symplectic potential near the
boundary of the moment polytope. Consider any constant metric $G=\t G > 0$ on
$\R^n$. For an arbitrary closed convex polyhedral set $P\subset \R^n$ and any point
$p\in\partial P$, denote by $\CC_{p}^G$ the closed cone of directions
that are {}``outward pointing at $p$'' in the following sense,
\[
 \CC_{p}^G(P) := \{ c\in\R^{n}:\t c G (p-p')\geq0, \quad\forall p'\in P\}.
\]
Notice that for the Euclidean metric $G=I$, the cone of $p$ is precisely the
negative of the cone of the fan of $P$, corresponding to the face $p$ lies in;
in this case we will write $\CC_p := \CC_p^I$.

\begin{lem} \label{grad_boundary} For any sequence $x_k\in\iP$ that
  converges to a point in the boundary, $x_k\to p\in\partial P$, we have
\[
 \pd{(g_P+\varphi)}{x}\vert_{x_k} \to \CC_p(P) ,
\]
in the sense that for any $c\notin \CC_p(P)$ there is an open
neighborhood $U\ni p$ such that
\[
 \R_0^+ \pd{(g_P+\varphi)}{x}\vert_x \neq \R_0^+ c \qquad \forall x\in U\cap\iP .
\]
\end{lem}
\begin{proof}
Suppose (using an affine change of coordinates $l(x) = Ax-\lambda$)
that $p$ lies in the codimension $k$ face, $k>0$, where $l_{1}=\dots=l_{k}=0$
and $l_{j}>0$ for $j=k+1,\dots,n$. We have
\[
 \CC_p = \{c\in\R^n: \t c A^{-1}(l(p)-l(p')) \geq 0, \forall p'\in P\} ,
\]
and $l_i(p-p') \leq 0$ for $i=1,\dots,k$, whereas there is no
restriction on the sign of $l_j(p-p')$ for $j=k+1,\dots,n$. Therefore,
\[
 c\in\CC_p \iff c= \t A \widetilde{c} \textrm{ with }
 \widetilde{c}\in (\R_0^-)^k\times\{0\} \subset \R^n .
\]
Since
\[
 \pd{(g_P+\varphi)}{x} = \t A \pd{(g_P+\varphi)}{l}
\]
it is therefore sufficient to prove that $\pd{(g_P+\varphi)}{l}$
approaches $(\R_0^-)^k\times\{0\}\subset\R^n$ as we get near $p$. Indeed,
we find that
\[
 \pd{g_{P}}{l}=\frac{1}{2}\pd{}{l}\sum_{a=1}^{d} \ell_{a}\log
  \ell_{a}=\frac{1}{2}\pd{}{l}\big(\sum_{i=1}^{k} l_{i}\log l_{i}+
 \sum_{j=k+1}^{n} l_{j}\log l_{j}+\sum_{m=n+1}^{d}\ell_{m}\log\ell_{m}\big)
\]
and hence
\begin{equation} \label{dgPdl}
 \pd{g_{P}}{l_{r}}=\frac{1}{2}(1+\log l_{r}+
\sum_{m=n+1}^{d}\pd{\ell_{m}}{l_{r}}(1+\log \ell_{m})).
\end{equation}
For $m>n$ (actually, for $m>k$), the sum is bounded in a neighborhood of
$p$ since $\pd{\ell_{m}}{l_{r}}$ is constant and $\ell_{m}>0$ at $p$.
Since $\pd{\varphi}{l_r}$ is also bounded for any $r$, $\pd{(g_{P}+\varphi)}{l_{j}}$
is bounded in a neighborhood of $p$ for $j=k+1,\dots,n$ and clearly
$\pd{(g_{P}+\varphi)}{l_{i}}\to-\infty$ for $i=1,\dots,k$ as we
approach $p$, which proves the lemma.
\end{proof}

In the following lemma, we relate the Legendre transforms $\kappa_s$ and $\L_\psi$
at large $s$, more precisely:

\begin{lem} \label{nbhd} For any two points $p\neq p'\in P$, there exist $\varepsilon>0$
and $s_0\geq 0$ that depend only on $p'$ and the distance $d(p,p')$, such that for all $s\geq s_0$
\[
 \overline{\kappa_s(B_\varepsilon(p')\cap\iP)} \cap \big( \L_\psi(p)+\CC_p(P) \big)
 = \emptyset .
\]
\end{lem}
\begin{proof} We will show that there is a hyperplane separating $\kappa_s(B_\varepsilon(p')\cap\iP)$
and $\L_\psi(p)+\CC_p(P)$; we continue to use a chart as in Lemma \ref{grad_boundary}.

For any two points $p\neq p'$,
\begin{eqnarray*}
  \t (l(p)-l(p')) (\pd{\psi}{l}\vert_p - \pd{\psi}{l}\vert_{p'}) & = &
  \t (l(p)-l(p')) \int_0^1 (\Hess_{p'+\tau(p-p')}\psi) \dd \tau (l(p)-l(p')) \geq \\
 & \geq & c_\psi \| l(p)-l(p') \| ^2 \ > \ 0
\end{eqnarray*}
where $c_\psi>0$ is a constant depending only $\psi$. This implies
that there is at least one index $j$ such that $l_j(p)\neq l_j(p')$,
\[
 \sgn (\pd{\psi}{l_j}\vert_p - \pd{\psi}{l_j}\vert_{p'}) = \sgn (l_j(p)-l_j(p')),
\]
and also
\[
 \left| \pd{\psi}{l_j}\vert_p - \pd{\psi}{l_j}\vert_{p'} \right| \geq \frac{c_\psi}{n} | l_j(p)-l_j(p')|.
\]
Choose $\varepsilon > 0$ small enough (this choice depends on $\psi$ only) so that
\[
 \pd{\psi}{l_j}\vert_p \notin \pd{\psi}{l_j}(B_\varepsilon(p')\cap\iP)
\]
and consider first the case that $0\leq l_j(p) < l_j(p')$. For all
$x\in B_\varepsilon(p')\cap\iP$, from equation (\ref{dgPdl}),
\[
 \pd{g_P}{l_j}\vert_x \geq c_1 \log^- (l_j(p')-\varepsilon)+c_2 = C,
\]
where $\log^-$ denotes the negative part of the logarithm, and $c_1,c_2,C$ are 
constants depending only on $p'$ and $\varepsilon$. Then, for any $\delta > 0$ such that
\[
 \pd{\psi}{l_j}\vert_p+\delta \notin \pd{\psi}{l_j}(B_\varepsilon(p')\cap\iP),
\]
we find $s_0 = \frac{2|C|}{\delta}$ so that
\[
 \pd{\psi}{l_j}\vert_x+\frac 1s \pd{g_P}{l_j}\vert_x \geq \pd{\psi}{l_j}(p)+\frac{\delta}{2} ,
 \qquad \forall x\in B_\varepsilon(p')\cap\iP .
\]
Hence, also
\[
 \pd{(\psi+\frac 1s \varphi)}{l_j}\vert_x+\frac 1s \pd{g_P}{l_j}\vert_x \geq \pd{\psi}{l_j}(p)+\frac{\delta}{4} ,
 \qquad \forall x\in B_\varepsilon(p')\cap\iP,
\]
for $s$ big enough (the additional condition depending only on $\varphi$, which is globally controlled
on the whole polytope $P$), which proves our assertion.

If, on the other hand, $0\leq l_j(p') < l_j(p)$, we see again from equation (\ref{dgPdl})
that
\[
 \pd{g_P}{l_j} \leq c'_1 \log^+ l_j+c'_2
\]
on $B_\varepsilon(p')\cap\iP$, where $\log^+$ stands for the positive part of the logarithm. Again,
\[
 \pd{g_P}{l_j}\vert_x \leq c'_1 \log^+ (l_j(p')+\varepsilon)+c'_2 = C',  \qquad  \forall x\in B_\varepsilon(p')\cap\iP
\]
and the same argument applies.
\end{proof}

We will now characterize the limit amoeba in terms of the tropical
amoeba via a projection $\pi$ that can be described as follows.

\begin{lem} \label{proj} For any strictly convex $\psi$ as above, there exists a
partition of $\R^n$ indexed by $P$ of the form
\begin{equation} \label{part}
 \R^n = \coprod_{p\in P} \L_\psi(p)+\CC_p(P) .
\end{equation}
In particular, there is a well-defined continuous projection $\pi:\R^n\to\L_\psi P$
given by
\[
 \pi(\L_\psi(p)+\CC_p(P)) = \L_\psi(p) .
\]
\end{lem}

\begin{proof}
It suffices to show that for $p\neq p'$,
\[
 \big( \L_\psi(p)+\CC_p(P)\big) \cap \big(
 \L_\psi(p')+\CC_{p'}(P)\big) = \emptyset .
\]
To see this, assume that
\[
 \L_\psi(p)+c = \L_\psi(p')+c', \textrm{ with } c\in\CC_p(P),c'\in\CC_{p'}(P) .
\]
Then $\t(c-c')(p-p')\geq 0$, from the definition of the cones; on the other hand,
\begin{eqnarray*}
 \t (p-p') (c-c') & = & \t (p-p') (\L_\psi(p')-\L_\psi(p)) = \\
 & = & \t (p-p') \int_0^1 (\Hess_{p+\tau(p'-p)}\psi) \dd \tau (p'-p) < 0,
\end{eqnarray*}
which is a contradiction.
\end{proof}

\begin{rem} For quadratic $\psi$ with  $\t G = G > 0$ symmetric and
positive definite there is a more intrinsic description of the map $\pi$:
it is given by the projection of $\R^{n}$ on the closed
convex subset $P$ under which each point projects onto its best
approximation in the polytope $\L_\psi P$ with respect to the metric
$G^{-1}$ (see, for instance, chapter \textsc{v} of \cite{Bou}, and also
Figure \ref{fig1}),
\[
 p = \pi(y) \iff p\in \L_\psi P \land \forall p'\in \L_\psi P\setminus\{p\}:
  \| y-p \|_{G^{-1}} < \| y-p' \|_{G^{-1}}.
\]
Note also that
\[
\forall y\in\R^{n}: y-\pi(y) \in \CC_{\pi(y)}^{G^{-1}}(\L_\psi P)
\]
and that, in fact, $\pi(y)$ is characterised
by this property, i.e.
\[
 \forall p\in \L_\psi P: y-p\in\CC_{p}^{G^{-1}}(\L_\psi P) \iff \pi(y)=p.
\]
In this sense, $\CC_p^{G^{-1}}(\L_\psi P)$ is a kind of ``convex kernel at $p$'' of the
convex projection $\pi$.
Note, by the way, that in this case $\CC_x(P) = \CC_{\L_\psi(x)}^{G^{-1}}(\L_\psi P)$.
\end{rem}

\begin{figure}
 \begin{centering}
 \includegraphics[width=9cm,height=7cm,keepaspectratio,angle=-90]{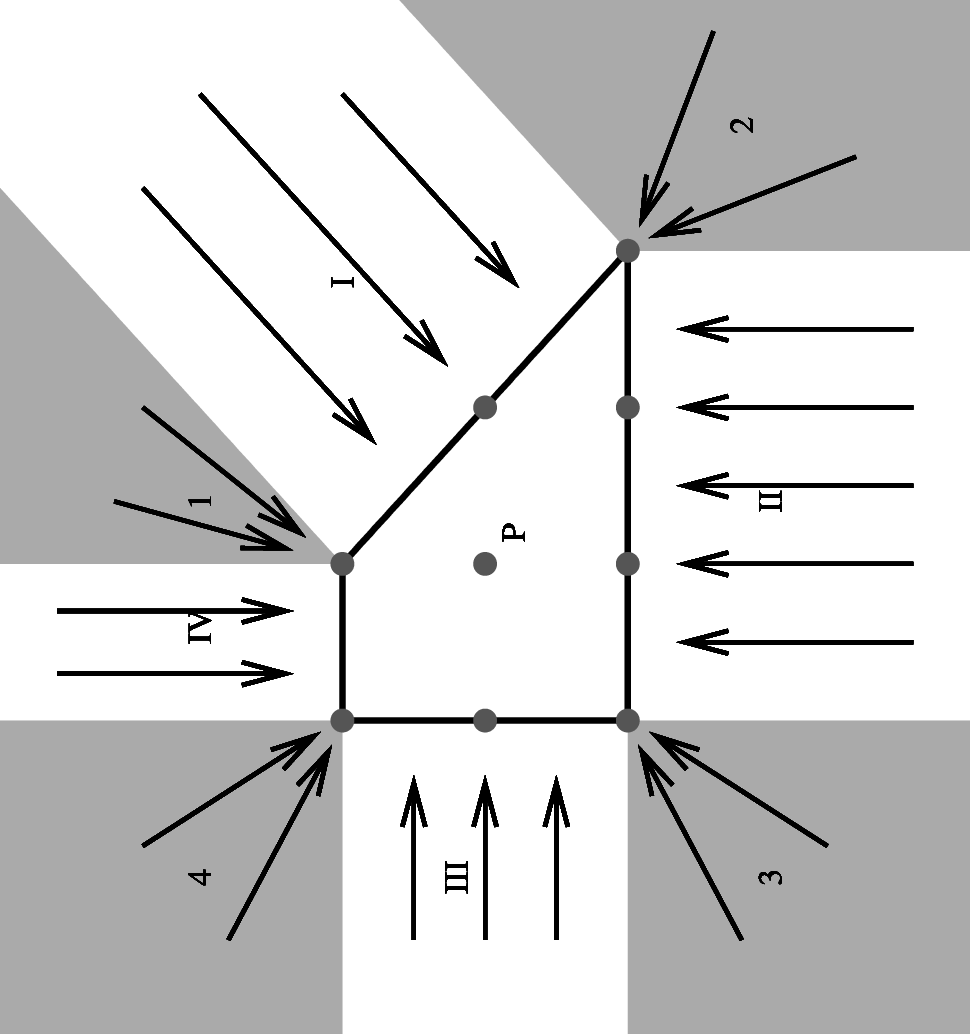}
 \par\end{centering}
 \caption{The map $\pi$.}
 \label{fig1}
\end{figure}

Finally, we have

\begin{proofb}{\ref{amoeba_thm}}
We first show that $\L_\psi \circ \kappa_s^{-1} \to \pi$ pointwise on
$\R^n$. For points in the interior of $\L_\psi P$, where $\pi\vert_{\iP} =
\mathrm{id}_{\iP}$, this is clear from Lemma \ref{k_conv}.

Consider, therefore, any point $y\notin \L_\psi\iP$, its family of
inverse images $x_s = \kappa_s^{-1}(y)\in\iP$, and any convergent
subsequence $x_{s_k}\to p$. Then the limit lies in the boundary,
$p\in\partial P$. We need to show that $\L_\psi(p) = \pi(y)$, or, what
is the same, that $y-\L_\psi(p) \in \CC_{p}(P)$. This is guaranteed by Lemma \ref{grad_boundary},
\[
 \frac 1s \pd{(g_P+\varphi)}{x}\vert_{x_{s_k}} = (\kappa_s-\L_\psi)(x_{s_k}) = y-\L_\psi(x_{s_k})\to
 y-\L_\psi(p) \in \CC_p(P) ,
\]
and proves pointwise convergence.

Now we can use compactness of $P$ (and hence of the space of closed
non-empty subsets of $P$ with the Hausdorff metric) to show the
result. Throughout the proof we will not distinguish between sets and
their closures since the Hausdorff topology does not separate them.

Let us first show that
\[
 \cdiag{
  \kappa_{s}^{-1}\Atrop \ar[r]^-{\mathrm{H}} & \L_\psi^{-1}\circ\pi\Atrop
 } .
\]
Take any convergent subsequence $\cdiag{\kappa_{s_k}^{-1}\Atrop \ar[r]^-{\mathrm{H}} & K \subset P}$;
since $\kappa_s^{-1}\to\L_\psi^{-1}\circ\pi$ pointwise, it follows that
\[
 K \supset \L_\psi^{-1}\circ\pi\Atrop .
\]
{}For the other inclusion, consider any point $p'\notin \L_\psi^{-1}\circ\pi\Atrop$; since
the distance of $p'$ to $\L_\psi^{-1}\circ\pi\Atrop$ is strictly positive, by Lemma
\ref{nbhd} there is a neighborhood $U$ of $p'$ in $\iP$ and a $s_0$ such that for all
$p\in\L_\psi^{-1}\circ\pi\Atrop$ and $s\geq s_0$, the sets $\kappa_s(U)$ and
$\L_\psi(p)+\CC_p(P)$ not only have empty intersection but are actually separated
by a hyperplane. But this implies, in particular, that for $s$ large enough
\[
 U\cap \kappa_s^{-1}\Atrop = \emptyset
\]
and $p'\notin K$, as we wished to show.

In the last step, we prove that $\cdiag{\kappa_{s}^{-1}\A_s \ar[r]^-{\mathrm{H}} &
\L_\psi^{-1}\circ\pi\Atrop}$. Again, using compactness, it is sufficient to
consider any convergent subsequence $\cdiag{\kappa_{s_k}^{-1}\A_{s_k} \ar[r]^-{\mathrm{H}} & K'}$.

To show that $\L_\psi^{-1}\circ\pi\Atrop \subset K'$, it is sufficient to observe that
$\Atrop\subset\A_{s_k}$ (see \cite{GKZ,Mi}) and hence $\kappa_{s_k}^{-1}\Atrop \subset
\kappa_{s_k}^{-1}\A_{s_k}$.

For the converse inclusion $K'\subset \L_\psi^{-1}\circ\pi\Atrop$, consider the constant $c$
from inequality (\ref{unif_inj}) above, and set
\[
 \varepsilon_k := \frac 1c \dist(\A_{s_k},\Atrop) .
\]
This sequence converges to zero, and therefore the closed
$\varepsilon_k$-neighborhoods
$(\kappa_{s_k}^{-1}\Atrop)_{\varepsilon_k} \supset
\kappa_{s_k}^{-1}\Atrop$ still converge to $\L_\psi^{-1}\circ\pi\Atrop$. But as $\kappa_{s_k}$
satisfies the uniform bound in (\ref{unif_inj}),
\[
 \kappa_{s_k}((\kappa_{s_k}^{-1}\Atrop)_{\varepsilon_k}) \supset
 (\Atrop)_{c\varepsilon_k} \supset \A_{s_k}
\]
and hence
\[
 (\kappa_{s_k}^{-1}\Atrop)_{\varepsilon_k} \supset
 \kappa_{s_k}^{-1} \A_{s_k}
\]
which proves the second inclusion.
\end{proofb}

Low-dimensional examples of the relation between tropical and limit amoebas are illustrated in Figures \ref{fig112} to \ref{fig18} below. In the following remarks, we collect basic facts about limit amoebas and their relation to their tropical counterparts.
\begin{figure}
\[
\begin{centering}
 \begin{array}{c}
 \xymatrix@C=-2.5cm{
 \includegraphics[width=5cm,keepaspectratio]{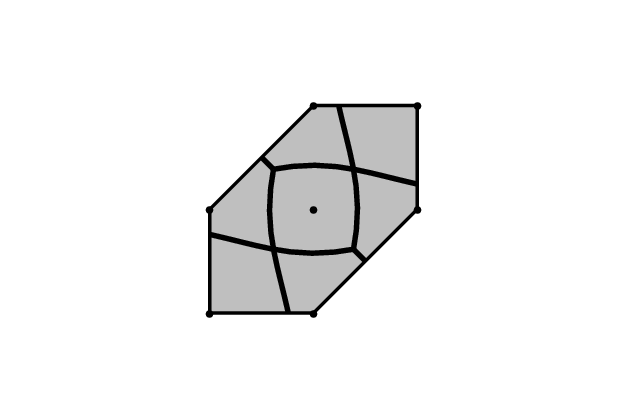} \ar[rd]_-{\L_\psi} & &
 \includegraphics[width=5cm,keepaspectratio]{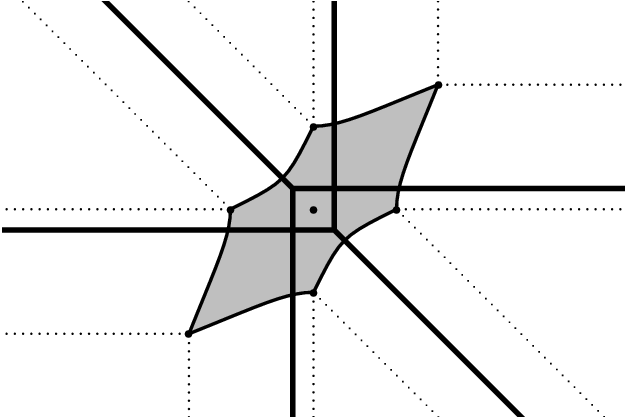} \ar[ld]^-{\pi}   \\
 & \includegraphics[width=5cm,keepaspectratio]{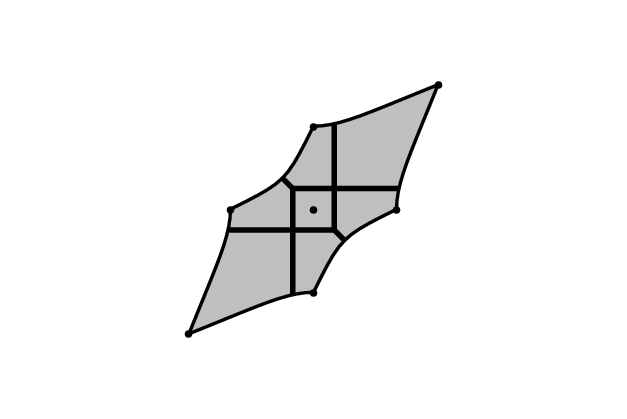}  & \\
 }
 \end{array}
 \end{centering}
\]
\caption{The situation of the main theorem (for $\psi$ not quadratic): $P\supset\Alim$ (top left), $\L_\psi P$ with the cones
of projection $\CC_p$ (dotted) and the tropical amoeba $\Atrop$ (top right) and $\L_\psi P \supset \pi\Atrop = \L_\psi \Alim$
(bottom).}
\label{fig112}
\end{figure}

\begin{rem} \begin{itemize} \item[(i)] \label{ex1} The first fact to catch the eye about the limit amoebas is that they depend on more parameters than the tropical amoebas: while the latter are determined by the valuation $v(m)$, the former vary heavily, depending on the direction $\psi$ of the geodesic ray $g_P+\varphi+s\psi$ we follow. This reflects the fact that we look at the family of hypersurfaces in different categories: while the complex biholomorphism class of the hypersurface $Y_s \subset X_P$ is independent of the K\"ahler metric we put on $X_P$, the Hausdorff limit of $\mu_P(Y_s) \subset P$ does vary substantially. This is illustrated for the simplest possible example, $\mathbb{P}^2$, with moment polytope the standard simplex in $\R^2$ and valuation $v(0,0)=0$, $v(1,0)=\frac 12$,
$v(0,1)=\frac 14$, in Figure \ref{fig2}.
\begin{figure}[!h]
\[
 \begin{centering}
 \begin{array}{c}
 \xymatrix@C=0cm{
\includegraphics[width=3.2cm,keepaspectratio]{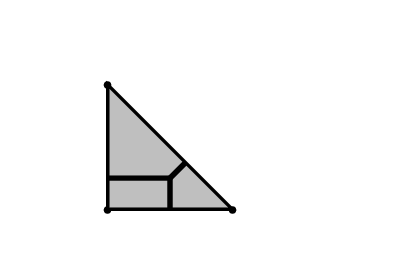} \ar[d]^-{\L_\psi=\mathrm{id}} &
\includegraphics[width=3.2cm,keepaspectratio]{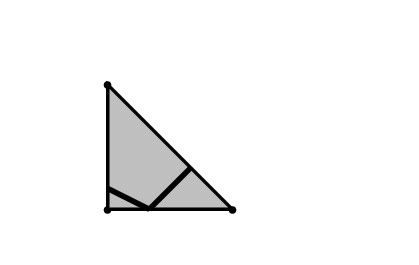}
 \ar[d]_-{\L_\psi=\frac 14\left[\begin{array}{@{}cc@{}} 6 & 3 \\ 3 & 6 \end{array}\right]} &
\includegraphics[width=3.2cm,keepaspectratio]{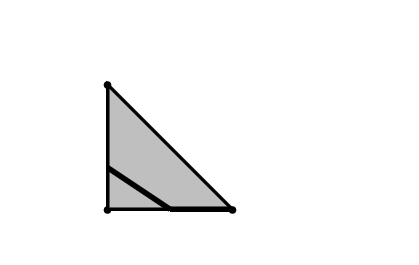}
 \ar[d]_-{\L_\psi=\frac 14\left[\begin{array}{@{}cc@{}} 3 & 2 \\ 2 & 3 \end{array}\right]} \\
\includegraphics[width=3.2cm,keepaspectratio]{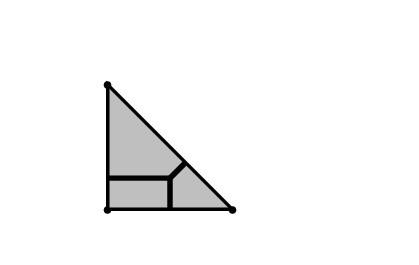} &
\includegraphics[width=3.2cm,keepaspectratio]{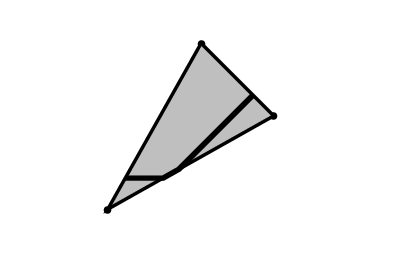} &
\includegraphics[width=3.2cm,keepaspectratio]{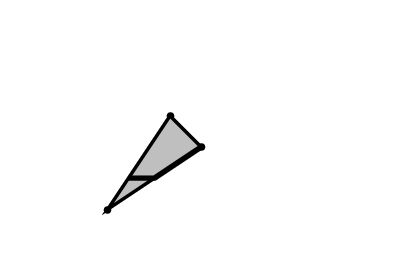} \\
\includegraphics[width=3.2cm,keepaspectratio]{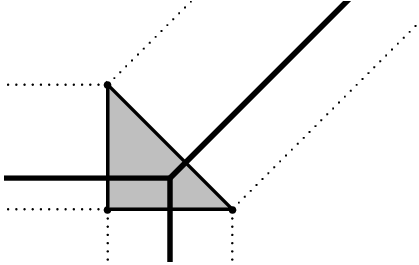} \ar[u]_-{\pi} &
\includegraphics[width=3.2cm,keepaspectratio]{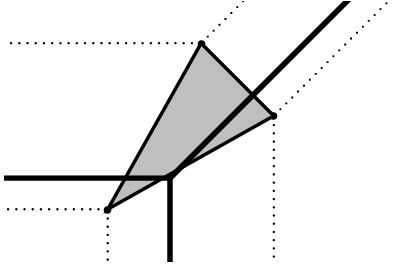} \ar[u]^-{\pi} &
\includegraphics[width=3.2cm,keepaspectratio]{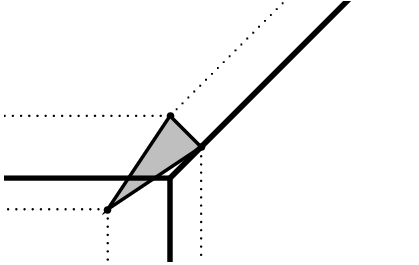} \ar[u]^-{\pi} \\
 }
 \end{array}
 \end{centering}
\]
\caption{Limit amoebas ($\Alim\subset P$, top row), their image under $\L_\psi$, and tropical amoebas ($\Atrop$, bottom row)
for different quadratic $\psi$ and fixed valuation $v$.}
\label{fig2}
\end{figure}
\item[(ii)] It is clear from this example already that it is not, in general, possible to recover the tropical amoeba from the limit amoeba, although we have (from the proof of the theorem) that always
\[
 \L_\psi \left( \Alim \cap \breve P \right) = \Atrop \cap \L_\psi \breve P .
\]
(For example, if $\psi$ is quadratic (thus $\L_\psi$ linear), the limit amoeba itself will be piecewise linear).
\item[(iii)] There is, however, an open set of valuations $v$ and directions $\psi$ such that the projection $\pi\Atrop$ will coincide with $\L_\psi P\cap \Atrop$;
this happens whenever the ``nucleus'' (i.e. the complement of all unbounded hyperplane pieces) of the tropical amoeba $\Atrop$ lies inside $\L_\psi P$,
since the infinite legs run off to infinity along directions in the cone of the
relevant faces. This situation is depicted in Figure \ref{fig112} for a limit of elliptic curves
in a del Pezzo surface,
\begin{eqnarray*}
 P & = & \langle (1,0),(1,1),(0,1),(-1,0),(-1,-1),(0,-1) \rangle \\
 \psi(x) & = & \frac{x^2}{2}+\frac{\|x\|^4}{4} \\
 v(m) & = & \frac{m^2}{2} .
\end{eqnarray*}
\end{itemize}
\end{rem}

At this point, naturally the question arises how much of the information encoded by tropical amoebas can be recovered from the compact limit amoebas, which we turn to now.

\subsection{Compact amoebas and enumerative information} Even without establishing a precise criterion for when the limit amoeba permits the recovery of the tropical amoeba, we can address these questions qualitatively.

\begin{prop}\begin{itemize} \item[(i)] For a fixed potential $\psi$, there is a set (with non-empty interior) of valuations such that the tropical amoeba can be recovered from the limit amoeba.
\item[(ii)] Conversely, for a fixed valuation $v$, there is a set (with non-empty interior) of potentials $\psi$ such that the tropical amoeba can be recovered from the limit amoeba.
\end{itemize}
\end{prop}
\begin{proof} Both assertions follow from the fact that scaling $\psi$ and $v$ (separately), the image of the polytope $\L_\psi P$ can be made arbitrarily big in relation to the tropical amoeba, since the tropical amoeba furthermore is determined once we reach its ``tentacles'' (the unbounded parts of hyperplanes).

In particular, the set of valuations in (i) contains all valuations such that $\L_\psi \Alim = \Atrop \cap P$.
\end{proof}

It is evident that for a fixed potential, only a bounded set of valuations will permit recovery of the tropical amoeba; for a fixed valuation or, actually, \emph{any bounded set} of valuations, any potential that is ``large enough'' will do. Applying this, for example, to the enumerative problem studied in \cite{Mi2,GM}, we immediately obtain:

\begin{cor} Let $P=\langle \vect{0,0}, \vect{d,0}, \vect{0,d} \rangle$. For a fixed set $S$ of $3d-1+g$ points in the plane in tropically generic position, there is a set (with non-empty interior) of potentials $\psi$ on $P$ such that the set of tropical curves of genus $g$ through $S$ is in bijective correspondence with the set of limit amoebas of genus $g$ through $\L_\psi^{-1} S$ in $P$.
\end{cor}

\subsection{Implosion of polytopes versus explosion of fans}

For simplicity, in the present subsection, we restrict ourselves to
the case $\psi(x) = \frac 12 x^2$. 
We remark that while the map $\pi$ projects onto $P$, the map $id-\pi$ 
is injective in the interior of the cones $v+\CC_v$ for all vertices
$v\in P$
(regions $1-4$ in figure \ref{fig3}).
For a face $F$ of dimension $k>0$ of $P$, the region $\breve F+\CC_p$, 
for any $p\in \breve F$, implodes to the cone $\CC_p$ of codimension
$k$. In particular, the polytope $P$ implodes to the origin.
\begin{figure}
\[
\begin{centering}
 \begin{array}{c}
 \xymatrix@C=-2cm{
 & \includegraphics[width=3.5cm,angle=-90]{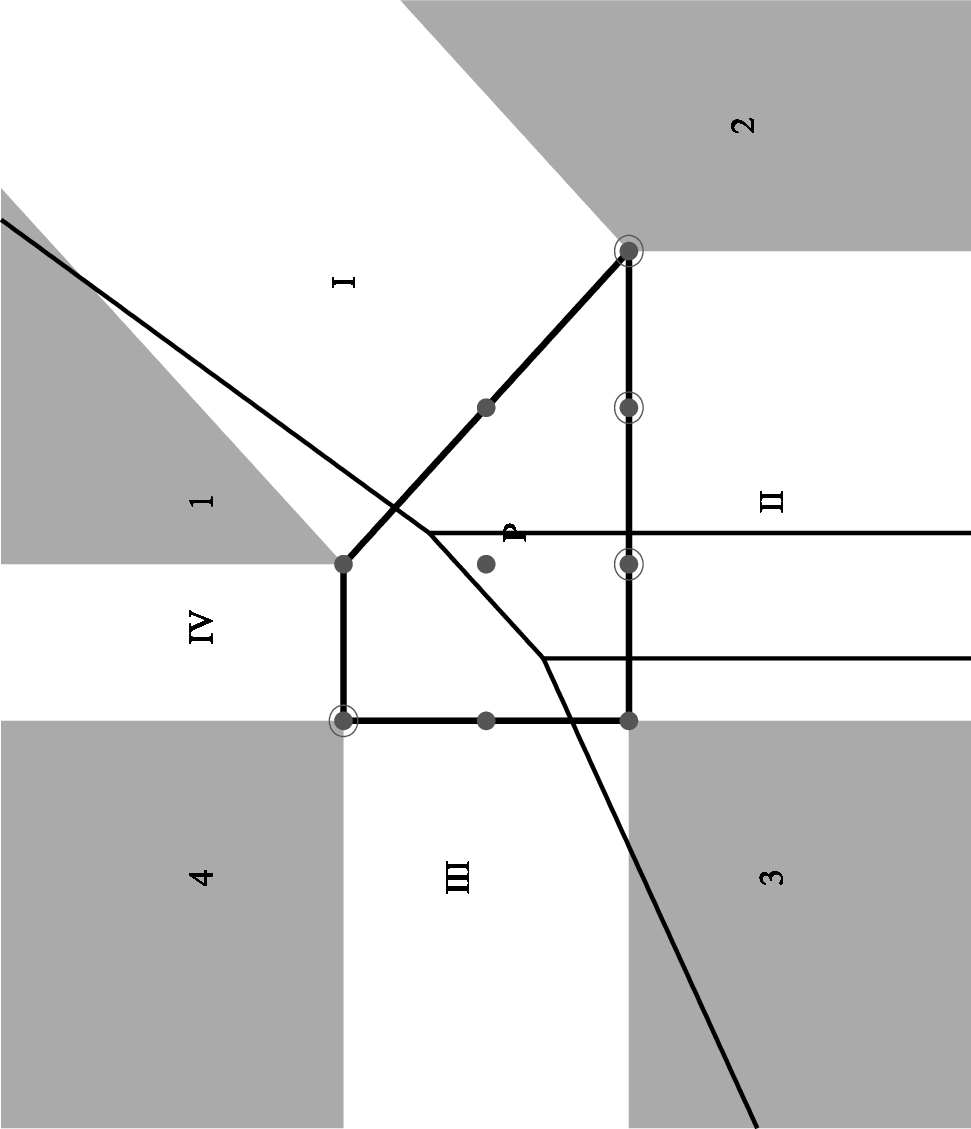}
 \ar[rd]^{\mathrm{id}-\pi} \ar[ld]_{\pi} & \\
\includegraphics[width=3.5cm,bb=200bp 240bp 404bp 519bp,angle=-90]{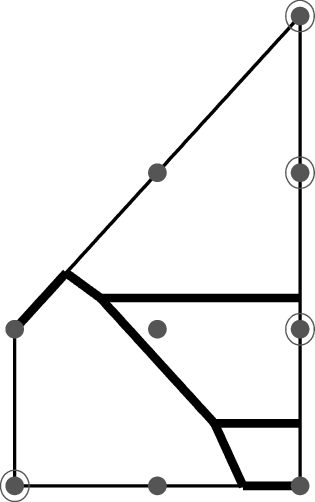} & &
\includegraphics[width=3.5cm,angle=-90]{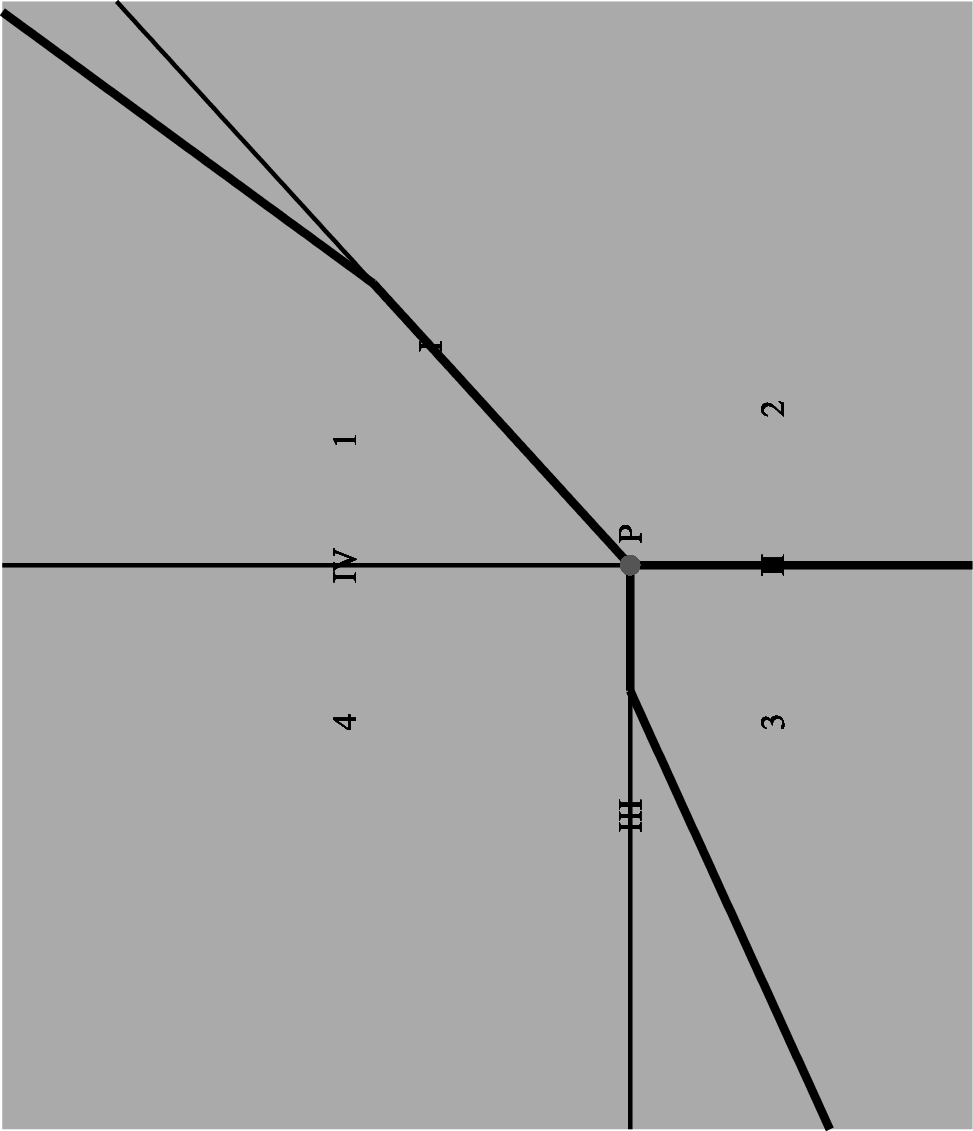} \\
 }
 \end{array}
\end{centering}
\]
\caption{Explosion of the fan and implosion of the polytope. The marked points
correspond to non-zero coefficients $a_m$.}
\label{fig3}
\end{figure}
Dually, the map $id-\pi$ explodes the fan along positive 
codimension cones. In particular, the origin is exploded to $P$.
In Figure \ref{fig3}, we consider the non-generic 
Laurent polynomial 
\[
a_1 + a_2 t^{-0.6} x + a_3 t^{-0.4}x^2+a_4 t^{1.8} \frac{y^2}{x},
\] 
its tropical amoeba for $\mathbb{P}^2$ blown-up at one point and the corresponding 
images under the maps $\pi$ and $id-\pi$. Note that, for this
non generic polynomial, part of the compact amoeba $\pi\Atrop$ lies in 
the boundary of $P$. Only in $\breve P$ does $\pi\Atrop$ coincide with the
tropical non-Archimedean amoeba.

\subsection{Amoebas associated to geometric quantization}

\renewcommand{\labelenumi}{\bf\theenumi}
\renewcommand{\theenumi}{\alph{enumi}.}
As we could observe in Remark \ref{ex1}, the behavior of the limit amoeba defined
via a fixed valuation is rather unstable. Actually, it does not only depend on the
choice of $\psi$, but also behaves badly under integer translations of the
moment polytope $P$. These shortcomings are somehow overcome by a specific choice of
valuation that is associated with a toric variety and a large K\"ahler structure limit
with quadratic $\psi$. This construction provides the natural link between the
convergence of sections to delta distributions considered in Section \ref{lcsl} that
comes out of geometric quantization, and the consideration of the image of hypersurfaces
defined by the zero locus of generic sections.

Geometric quantization motivates considering the hypersurfaces defined by
\[
 \breve{Y}_s^{\rm{GQ}} = \{p\in X_P: \sum_{m\in P\cap\Z^n} a_m \xi^m_s(p) = 0 \},
\]
where $a_m\in\C^*$, and $\xi^m_s$ are the $L^1$-normalized holomorphic sections
converging to delta distributions, as in Section \ref{lcsl}. A simple estimate of the
order of decay of $\|\sigma^m_s\|_1$ as $s\to\infty$ gives, in the notation
of Lemma \ref{convex},
\begin{eqnarray*}
  \|\sigma^m_s\|_1^{-1} \frac{\dd}{\dd s} \|\sigma^m_s\|_1 & = & 
  \|\sigma^m_s\|_1^{-1} \frac{\dd}{\dd s} \int\limits_P \e^{-h^s_m(x)} \dd x = \\
  & = & \| \e^{-h^s_m} \|_1^{-1} \frac{\dd}{\dd s} \int\limits_P \e^{-h^0_m(x)-s f_m(x)} \dd x = \\
  & = & - \int\limits_P f_m(x) \frac{\e^{-h^s_m(x)}}{\| \e^{-h^s_m} \|_1} \dd x \to -f_m(m) = \psi(m) .
\end{eqnarray*}
Therefore, we call the limit of the family of amoebas $\mu_P(Y_s)$, where
\[
 \breve{Y}_{s}=\{ w\in(\C^{*})^{n}:\sum_{m\in P\cap\Z^{n}} a_{m} \e^{-s\psi(m)} w^{m}=0 \},
\]
the GQ limit amoeba, $\Alim^{\rm GQ}$. The fact
that for this choice of valuations, inspired by geometric quantization, the
limit amoeba keeps away from integral points in $P$ is consistent with the
convergence of the holomorphic sections $\xi^m_s$ to delta distributions
supported on the Bohr-Sommerfeld fibers corresponding to those integral
points.

The behavior of this ``natural'' (from the point
of view of geometric quantization) choice of valuation is
illustrated in Figure \ref{fig4} for different quadratic
$\psi$. Note, in particular, that in the case $G_2$ there are
parts of the tropical amoeba $\Atrop$ that lie outside
$\L_\psi P$ and get projected onto subsets of faces (with
non-empty interior in the relative topology). Below, we will
give a complete characterization of the GQ limit amoeba
in this situation.

\begin{figure}[!h]
\[
 \xymatrix@C=0.5cm@R=0.2cm{
\includegraphics[width=3cm,keepaspectratio]{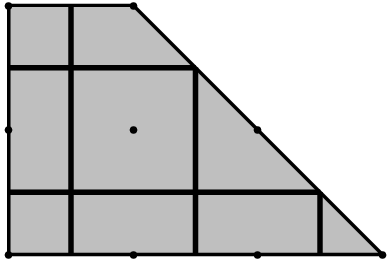}   &
\includegraphics[width=3cm,keepaspectratio]{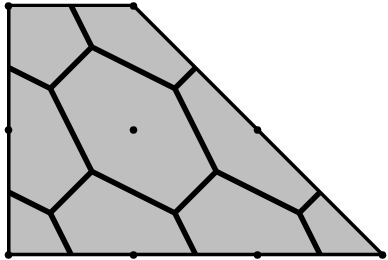}   &
\includegraphics[width=3cm,keepaspectratio]{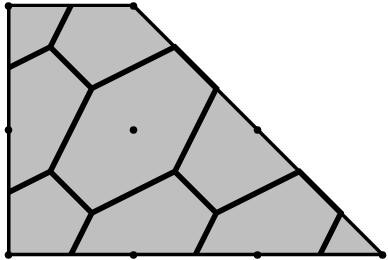} \\
 G_0 = \idmax & G_1 = \frac 14 \gmax & G_2=G_1^{-1} 
 \\
 }
\]
\caption{GQ amoebas associated to different quadratic $\psi$'s.}
\label{fig4}
\end{figure}

We start by observing that from the point of view
of equivalence of toric varietes described by different Delzant polytopes,
the GQ limit amoeba is well behaved:

\begin{prop} Let $\psi\in C^\infty_{\textrm{Hess} > 0}(P)$.
 \begin{enumerate}
  \item \label{GQtra} For any integer vector $k\in\Z^n$, setting $\widetilde{P}=P+k$ and
   $\widetilde{\psi}(\widetilde{x})=\psi(\widetilde{x}-k)$, we have
\[
 \widetilde{\A}_{\rm lim}^{\rm GQ} = \Alim^{\rm GQ}+k ,
\]
  \item \label{GQrot} For any base change $A\in Sl(n,\Z)$ of the lattice $\Z^n$, setting
  $\widetilde{P}=AP$ and $\widetilde{\psi}(\widetilde{x}) = \psi( A^{-1}\widetilde{x})$,
  we have
\[
 \widetilde{\A}_{\rm lim}^{\rm GQ} =  A \Alim^{\rm GQ} . 
\]
 \end{enumerate}
\end{prop}
\begin{proof} \ref{GQtra}: since $\widetilde{\psi}(\widetilde{m}=m+k)=\psi(m)$,
$\widetilde{\A}_{\rm trop}=\Atrop$; on the other hand,
\[
 \pd{\widetilde{\psi}}{x}\vert_{\widetilde{x}=x+k} = \pd{\psi}{x}\vert_x
\]
thus $\L_\psi P = \L_{\widetilde{\psi}} \widetilde{P}$ and
\[
 \L_\psi P \cap \Atrop = \L_{\widetilde{\psi}} \widetilde{P}\cap \widetilde{\A}_{\rm trop} .
\]

\ref{GQrot}: similarly, since $\widetilde{\psi}(\widetilde{m}=Am)=\psi(m)$ and
the tropical amoeba $\widetilde{\A}_{\rm trop}$ is defined via the functions
\[
 \widetilde{u} \mapsto \t \widetilde{m} \widetilde{u}-\widetilde{\psi}(\widetilde{m})
 = \t m \t A \widetilde{u}-\psi(m)
\]
it follows that $\widetilde{\A}_{\rm trop} = \t A^{-1}\Atrop$; for the Legendre transforms
one finds
\[
 \pd{\widetilde{\psi}}{x}\vert_{\widetilde{x}=Ax} = \t A^{-1} \pd{\psi}{x}\vert_x
\]
which proves the second claim.
\end{proof}

Actually, as is to be expected from the convergence of the sections defining
$\Alim^{\rm GQ}$ to delta distributions, these amoebas never intersect lattice
points:

\begin{prop} For any strictly convex $\psi\in C^\infty_{\textrm{Hess} > 0}(P)$, the GQ amoeba $\Alim^{\rm GQ}$
stays away from lattice points in the interior of $P$, that is,
\[
 \Alim^{\rm GQ}\cap\iP\cap\Z^n = \emptyset .
\]
\end{prop}
\begin{proof} We consider the functions that were used to define the tropical amoeba,
\[
 \eta_m( u ) = \t m u - \psi( m ), \qquad \forall m\in P\cap\Z^n
\]
and observe that for $u=u_m=\L_\psi m=\pd{\psi}{x}\vert_m$ this is the value of the Legendre
transform $h$ at $u_m$,
\[
 \eta_m(u_m) = h(u_m)
\]
where $h(u) = \t x(u) u - \psi(x(u)$. To show that
\[
 \eta_m(u_m) > \eta_{\widetilde{m}}(u_m) \qquad \forall \widetilde{m}\in P\cap\Z^n, \
 \widetilde{m} \neq m
\]
we use convexity of the Legendre transform $u\mapsto h(u)$, namely
\begin{eqnarray*}
 \eta_m(u_m) = h(u_m) & > & h(u_{\widetilde{m}}) + \t (u_m-u_{\widetilde{m}})\pd{h}{u}\vert_{u_{\widetilde{m}}} = \\
 & = & \t \widetilde{m} u_{\widetilde{m}} -\psi(\widetilde{m}) + \t (u_m-u_{\widetilde{m}})\widetilde{m} = \eta_{\widetilde{m}}(u_m)
\end{eqnarray*}
where we used the fact that
\[
 \pd{h}{u}\vert_{u_{\widetilde{m}}} = \pd{h}{u}\vert_{\pd{\psi}{x}\vert_{\widetilde{m}}} = \widetilde{m} .
\]
\end{proof}

When $\psi$ is quadratic it is possible to characterize the GQ limit
amoeba completely in terms of the limit metric on $P$ only. Let $F_p$ denote
the minimal face containing any given point $p\in\partial P$. Note that for
any point $x\in P$,
\[
 x \in F_p \iff (x-p) \perp_G \CC^G_p
\]
and, more generally, for any $x\in P$ and $c\in\CC^G_p$
\begin{equation}
 \| x+c-p \|_G^2 = \| x-p \|_G^2 + \| c \|_G^2 -2\|x-p\|_G\|c\|_G\cos\alpha \label{pyth}
\end{equation}
where $\alpha = \angle_G(x-p,c) \geq \frac \pi2$ since $\t c G (p-x) \geq 0$
by definition of $\CC_p^G(P)$.

\begin{prop} Let $\psi(x) = \frac{\t x G x}{2}+\t b x$ with $\t G = G > 0$; then a point $p\in P$
belongs to $\Alim^{\rm GQ}$ if and only if one of the following conditions holds:
\begin{enumerate}
 \item \label{intsct} There are (at least) two lattice points 
$m_1\neq m_2\in P\cap\Z^n$ such that
\[
 \|p-m_1\|_G = \|p-m_2\|_G = \min_{m\in P\cap\Z^n}\{ \|p-m\|_G\} .
\]
 \item \label{squeezed} $p\in\partial P$ and 
the unique closest lattice point does not lie in the
face $F_p$.
\end{enumerate}
\end{prop}
\begin{rem} The two conditions are, evidently, mutually exclusive:
from the description of the map $\pi$ it is evident that the inverse image of the
intersection of the tropical amoeba $\Atrop$ with $\L_\psi P$ is always a subset of
$\Alim$. This is taken care of by \ref{intsct}, while \ref{squeezed} describes the
parts of the GQ limit amoeba that arise from parts of $\Atrop$ that are ``smashed
on the boundary'' by the convex projection. In particular, for points in the interior
condition \ref{squeezed} is irrelevant.
\end{rem}
\begin{proof} \ref{intsct}: For quadratic $\psi$
and $u=\L_\psi x =Gx+b$, take
\begin{eqnarray*}
 \eta_m(u) & = & \t m u - \psi(m) = \t m u - (\frac{\t mGm}{2}+\t bm) = \\
 & = & \big(\t (\L_\psi m)G^{-1}u-\t b G^{-1} u\big) - \frac 12\big( \|\L_\psi m\|_{G^{-1}}^2 - 
 \| b \|_{G^{-1}}^2 \big) = \\
 & = & \frac 12 \big( \|u-b\|^2_{G^{-1}} - \|u-\L_\psi m\|^2_{G^{-1}} \big) .
\end{eqnarray*}
Since the first term is independent of $m$, it is irrelevant for the locus
of non-differentiability that defines the tropical amoeba,
\begin{eqnarray*}
 \Atrop & = & C^0\!-{\rm loc} \big( u\mapsto \max_{m\in P\cap\Z^n}
 \{ \eta_m(u) \} \big) = \\
 & = & C^0\!-{\rm loc} \big( u\mapsto \frac 12 \|u-b\|^2_{G^{-1}} -
 \min_{m\in P\cap\Z^n} \{ \frac 12 \|u-\L_\psi m\|^2_{G^{-1}} \} \big) .
\end{eqnarray*}
Therefore, $u$ lies in the tropical amoeba if and only if there are two
distinct lattice points $m_1\neq m_2$ in $P$ such that
\[
 \|u-\L_\psi m_1\|_{G^{-1}} = \|u-\L_\psi m_2\|_{G^{-1}} \iff \|p-m_1\|_G = \|p-m_2\|_G
\]
for $u=\L_\psi p \in \L_\psi P$. Taking into account that
$\L_\psi\Alim \supset \Atrop \cap \L_\psi P$, this proves that condition \ref{intsct}
is necessary and sufficient for $u$ to belong to this intersection.
Thus, either $p$ satisfies \ref{intsct} or if it belongs to $\Alim$ then it belongs to $\Alim\setminus \Atrop$.

\ref{squeezed}: Fix $p\in\partial P$. First we show that if there is a unique nearest
lattice point, say $m_p$, that lies in the face of $p$, $F_p$, then $p\notin\Alim$. By
the definition of $\Alim$, we have to show that
\[
 \forall c\in\CC_p^G(P): \qquad \L_\psi p+c\notin\Atrop,
\]
which follows in particular if $\eta_{m_p}(\L_\psi p+c) > \eta_m(\L_\psi p+c)$ for any
lattice point $m\neq m_p$. By the reasoning above, this is equivalent to
\[
 \| \L_\psi p +c - \L_\psi m_p \|^2_{G^{-1}} < \| \L_\psi p+c - \L_\psi m \|^2_{G^{-1}} ,
\]
which follows straight from equation (\ref{pyth}) (with $\cos \alpha = 0$).

For the other implication, assume $m_p\notin F_p$. Then,
\[
 \eta_{m_p}(\L_\psi p) > \eta_m(\L_\psi p), \qquad \forall m\in P\cap\Z^n, m\neq m_p
\]
and it suffices to show that for any $m\in F_p\cap \Z^n$ and $c\in\CC_p^G(P)\setminus\{0\}$
\begin{equation}
 \| \L_\psi p +\tau c -\L_\psi m \|^2_{G^{-1}} < \| \L_\psi p +\tau c -\L_\psi m_p \|^2_{G^{-1}}
 \label{sqzeq2}
\end{equation}
for some $\tau > 0$ large enough. The left hand side equals 
\[
 \| \L_\psi p +\tau c -\L_\psi m \|^2_{G^{-1}} = \| p-m \|^2_G+\tau^2 \| c \|_{G^{-1}}^2
\]
whereas the right hand side gives
\[
 \| \L_\psi p +\tau c -\L_\psi m_p \|^2_{G^{-1}} = \| p-m_p \|^2_G+\tau^2 \| c \|_{G^{-1}}^2 -
  2 \tau \| p-m_p \|_G \| c \|_{G^{-1}} \cos \alpha
\]
where $\cos \alpha < 0$. Subtracting therefore the left hand side of inequality
(\ref{sqzeq2}) from the right hand side, we are left with the expression
\[
 \| \L_\psi p +\tau c -\L_\psi m_p \|^2_{G^{-1}} -  \| \L_\psi p +\tau c -\L_\psi m \|^2_{G^{-1}} = 
 -  2 \tau \| p-m_p \|_G \| c \|_{G^{-1}} \cos \alpha \to + \infty,
\]
as $\tau\to\infty$, which finishes the proof.
\end{proof}

\subsection{Relation to other aspects of degeneration of K\"ahler structures}
\label{hms}

Degenerating families of K\"ahler structures have been studied from a variety of
viewpoints. In this section we would like to briefly address aspects of the
relation of the present work to some of those.

The first link is to the occurence of torus fibrations in mirror symmetry,
following \cite{SYZ} (see also \cite{Au}). As described in Section \ref{sect_tangentcone}, the K\"ahler metrics along a geodesic ray $g_P+\varphi+s\psi$ collapse, when rescaled appropriately, to a Hessian metric on the moment polytope $P$. The metric and/or affine structure the limit amoeba $\Alim$ inherits for certain combinations of valuation $v(m)$ and direction $\psi$ could be of interest to the SYZ approach to mirror symmetry 
\cite{GW,KS}, in the case when $P$ is
reflexive. Even though the induced metric on the complex hypersurfaces $Y_s\subset X_P$ will not in general be
Ricci flat, it is not inconceivable that by carefully choosing the available
parameters one might obtain the desired asymptotic behaviour.

\begin{ex}\label{ex_quarticP4} Consider the tropical version of a quartic surface in $\mathbb{P}^3$, with moment polytope given by the tetrahedron
\[
 P = \langle \vect{-1,-1,-1}, \vect{3,-1,-1}, \vect{-1,3,-1}, \vect{-1,-1,3} \rangle .
\]
If we set the valuation to be $1$ on these vertices, $0$ on the origin, and sufficiently negative on the other lattice points in $P$, we obtain a tropical amoeba $\Atrop$ whose ``nucleus'' is a tetrahedron $Q$ with vertices
\[
 Q := \langle \vect{1,1,1}, \vect{-1,0,0}, \vect{0,-1,0}, \vect{0,0,-1} \rangle .
\]
Choosing $\psi$ quadratic corresponding to the matrix
\[
 \psi \sim \frac{1}{4} \left[ \begin{array}{ccc} 2 & 1 & 1 \\ 1 & 2 & 1 \\ 1 & 1 & 2 \end{array} \right]
\]
we obtain a transformed polyhedron $\mathcal{L}_{\psi} P = -Q$, and the image of the projection
$\pi\Atrop \subset \mathcal{L}_{\psi} P$ is the octahedron $O$ with vertices
\[
 O := \langle \pm\vect{\frac{1}{2},\frac{1}{2},0}, \pm\vect{\frac{1}{2},0,\frac{1}{2}}, \pm\vect{0,\frac{1}{2},\frac{1}{2}} \rangle .
\]
The situation is depicted in Figure \ref{fig17}.
\begin{figure}
 \begin{centering}
 \includegraphics[width=9cm]{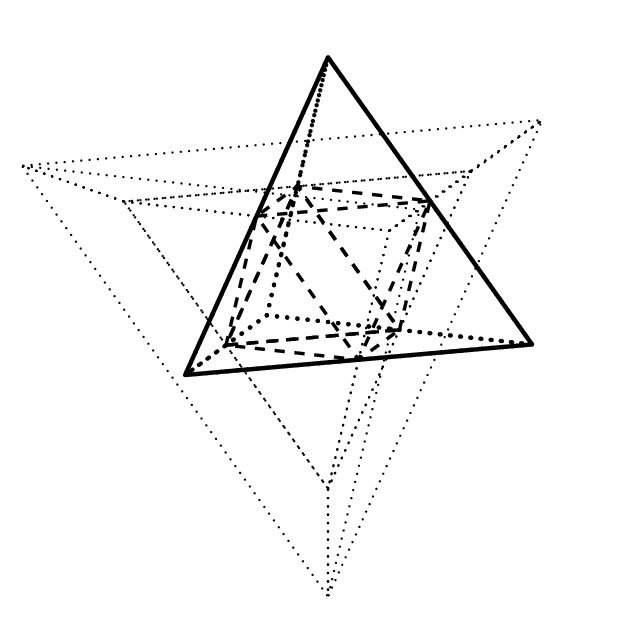}
 \end{centering}
 \caption{Example of the projection of a tropical amoeba for a quartic surface in $\mathbb{P}^3$, with $\Atrop$ dotted, $\mathcal{L}_{\psi} P$ solid, and $\pi\Atrop$ dashed.}
 \label{fig17}
\end{figure}
Note that from the construction and the drawing it is clear that the compact amoeba $\pi\Atrop$ inherits an affine structure (from the Legendre transformed coordinates on the moment polyhedron). It is, however, nonsingular even around the vertices.

There exists an entirely analogous example for the quintic in $\mathbb{P}^4$, which however is more difficult to draw; the tropical amoeba's ``nucleus'' is
\[
 Q = \langle \vect{1,1,1,1}, \vect{-1,0,0,0}, \vect{0,-1,0,0}, \vect{0,0,-1,0} \vect{0,0,0,-1} \rangle  = - \mathcal{L}_{\psi} P ,
\]
which again is symmetric to the Legendre transformed moment polyhedron. The (projection to the first three dimensions of the) intersection of $\pi\Atrop$ with a facet $F$ of $\mathcal{L}_{\psi} P$ is shown in Figure \ref{fig18}.
\begin{figure}
 \begin{centering}
 \includegraphics[width=5.5cm]{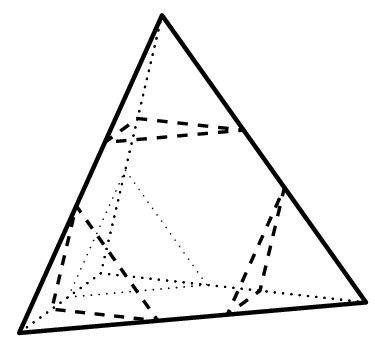}
 \end{centering}
 \caption{The interior of the dashed polyhedron corresponds to the intersection of $\pi\Atrop$ with a facet $F$ of $\mathcal{L}_{\psi} P$ for a quintic in $\mathbb{P}^4$.}
 \label{fig18}
\end{figure}
The affine structure apparently has singularities along the  line segments (twenty, overall) where $\pi\Atrop$ meets the edges of $\mathcal{L}_{\psi} P$.
\end{ex}

In a different context, in \cite{Pa1,Pa2} Parker considers degenerating families
of almost complex structures in an extension of the smooth category constructed
using symplectic field theory, to study holomorphic curve invariants. The typical
behaviour of his families of almost complex structures, depicted for the moment
polytope in the case of $\mathbb{P}^2$ in the introduction of \cite{Pa1}, is, in
the case of toric manifolds, remarkably reproduced in our approach. In fact, in
the notation of sections \ref{prelim} and \ref{largepol}, 
\[
\sum_{l,k=1}^n (G_0)_{jk}(G_0+s \Hess \psi)^{-1}_{kl} \frac{\partial}{\partial y_l} =
  \frac{\partial}{\partial y_j^s},
\]
where $J_s (\frac{\partial}{\partial \theta_j}) = \frac{\partial}{\partial y^s_j}$.
Therefore, in interior regions of $P$, where as $s\to\infty$ the term with
$\Hess \psi$ dominates, we have that the coordinates $y^s$ appear stretched relative
to the coordinates $y$ by an $y$-dependent transformation that scales with $s$. On
the other hand, as we approach a face $F$ of $P$ where some coordinates $l_j$ vanish,
the derivatives of $g_P$ with respect to these $l_j$'s will dominate, so that the
corresponding $y^s_{l_j}$'s do not scale with respect to the $y_{l_j}$'s. In the
directions parallel to $F$, however, we still have the term with $\Hess \psi$
dominating and for these directions the scaling with $s$ will occur. This is exactly
the qualitative behaviour described  in \cite{Pa1}. We note, however, that
in our approach this behaviour is implemented by deforming integrable toric complex
structures.

\vspace{1cm}
\textbf{Acknowledgements:} We wish to thank Miguel Abreu for many
useful conversations and for suggesting possible applications to tropical geometry. We would also like to thank the referees for thoughtful suggestions that led to substantial improvements in content and exposition.
This work is partially supported by the Center for 
Mathematical Analysis, Geometry and Dynamical Systems, IST, the Centro de Matem\'atica da Universidade do Porto, and by Funda\c c\~ao para a 
Ci\^encia e a Tecnologia through the Program POCI 2010/FEDER and by the projects 
POCI/MAT/58549/2004 and PPCDT/MAT/58549/2004. 
TB is also supported by Funda\c{c}\~ao para a Ci\^encia e a Tecnologia through the fellowships
SFRH/BD/22479/2005  and PTDC/MAT/098770/2008.


{\small \baselineskip 3.8mm}{\small \par}

\end{document}